\documentclass
{article}

\usepackage[colorlinks=true,breaklinks=true,citecolor=blue,linkcolor=blue,urlcolor=blue,unicode]{hyperref}
\title{Transfer of Fourier multipliers into Schur multipliers and sumsets in a discrete group}
\author{Stefan Neuwirth\and \'Eric Ricard}
\date{
}

\usepackage[english]{babel}
\usepackage{textcase,mathdots}
\usepackage[T1]{fontenc}
\usepackage{amsmath}
\usepackage{amssymb,mathrsfs,amsthm}
\theoremstyle{plain}

\newtheorem{thm}{Theorem}[section]

\newtheorem{lem}[thm]{Lemma}

\newtheorem{prp}[thm]{Proposition}
\newtheorem{cor}[thm]{Corollary}

\theoremstyle{definition}
\newtheorem{dfn}[thm]{Definition}

\newtheorem*{NotTer}{Notation and terminology}

\newtheorem{???}[thm]{Conjecture}

\theoremstyle{remark}
\newtheorem{rem}[thm]{Remark}
\newtheorem{exa}[thm]{Example}

\numberwithin{equation}{section}

\providecommand{\abs}[1]{\lvert#1\rvert}

\def\block#1#2{{\setbox0=\hbox{#1\kern1ex}\leftskip=\wd0\parindent=-\wd0\par\leavevmode\box0 #2\par}} 
\def\borne{\mathbb B}
\def\C{\mathbb C}
\DeclareMathOperator{\card}{\#}
\def\cb{\mathrm{cb}}

\def\col{c}
\def\Col{C}

\def\Cont{\mathrm C}

\def\e{\mkern1mu\mathrm e\mkern1mu}

\def\eps{\varepsilon}

\def\Fourier{\mathrm M}
\def\ga{\gamma}
\def\Ga{\varGamma}
\renewcommand{\ge}{\geqslant}

\def\Hardy{\mathrm H}
\def\Ht{\mathscr H}
\def\iu{\mkern1mu\mathrm i\mkern1mu}
\def\Id{\mathrm{Id}}
\def\id{\epsilon}
\def\imp{\Rightarrow}

\def\io[#1,#2]{\mathopen]#1,\allowbreak#2\mathclose[}
\def\iog[#1,#2]{\mathopen]#1,\allowbreak#2]}
\def\bigiog[#1,#2]{\bigl]#1,\allowbreak#2\bigr]}
\def\iod[#1,#2]{[#1,\allowbreak#2\mathclose[}
\def\bigiod[#1,#2]{\bigl[#1,\allowbreak#2\bigr[}
\def\Ell{\mathrm L}
\def\la{\lambda}
\def\La{\varLambda}
\def\bLa{{\breve\La}}
\def\HLa{{\mathaccent"707D \La}}
\renewcommand{\le}{\leqslant}
\def\mat{{\mathrm S}^\infty}

\def\N{\mathbb N}

\providecommand{\norm}[1]{\lVert#1\rVert}
\providecommand{\bignorm}[1]{\bigl\lVert#1\bigr\rVert}
\providecommand{\biggnorm}[1]{\biggl\lVert#1\biggr\rVert}
\providecommand{\Bignorm}[1]{\Bigl\lVert#1\Bigr\rVert}

\def\ph{\varphi}
\def\bph{{\breve\varphi}}
\def\Hph{{\mathaccent"707D \ph}}

\def\pstar{^{\scriptscriptstyle(\kern-1pt\lower0.5pt\hbox{$\scriptstyle *$}\kern-1pt)}}
\def\ppstar{^{\scriptscriptstyle(\kern-1pt*\kern-1pt)}}
\def\R{\mathbb R}
\def\rh{\varrho}

\def\Rt{\mathscr T}
\def\row{r}
\def\Row{R}
\def\Sch{\mathrm S}
\def\Schur{\mathrm M}
\DeclareMathOperator\sgn{sgn}
\DeclareMathOperator\Hsgn{\mathaccent"707D{sgn}}

\def\T{\mathbb T}
\newcommand{\tens}{\otimes}

\def\th{\vartheta}

\DeclareMathOperator{\tr}{tr}

\def\un{\chi}
\def\Hun{\mathaccent"707D\un_{\Z^+}}

\def\vn{\Ell^\infty}

\def\Z{\mathbb Z}

\begin{document}
\maketitle
\makeatletter
\renewcommand{\@makefntext}[1]{#1}
\makeatother
\footnotetext{Both authors were partially supported by ANR grant 06-BLAN-0015.}
\begin{abstract}
\noindent We inspect the relationship between relative Fourier multipliers on noncommutative Le\-besgue-Orlicz spaces of a discrete group~$\Ga$ and relative Toeplitz-Schur multipliers on Schatten-von-Neumann-Orlicz classes. Four applications are given: lacunary sets, unconditional Schauder bases for the subspace of a Lebesgue space determined by a given spectrum $\La\subseteq\Ga$, the norm of the Hilbert transform and the Riesz projection on Schatten-von-Neumann classes with exponent a power of~2, and the norm of Toeplitz Schur multipliers on Schatten-von-Neumann classes with exponent less than~1.
\end{abstract}

\section{Introduction}

\noindent Let $\La$~be a subset of~$\Z$ and let $x$~be a bounded measurable
function on the circle~$\T$ with Fourier spectrum in~$\La$:
we write $x\in\Ell^\infty_\La$, $x\sim\sum_{k\in\La}x_kz^k$. The matrix of the associated operator
$y\mapsto xy$ on~$\Ell^2$ with respect to its trigonometric basis is the Toeplitz matrix
\begin{equation*}
  (x_{r-c})_{(r,c)\in\Z\times\Z}=
  \bordermatrix{%
    &\scriptstyle\cdots&\scriptstyle1&\scriptstyle0&\scriptstyle-1&\scriptstyle\cdots\cr
    \hfil\scriptstyle\vdots&\ddots&\ddots&\ddots&\ddots&\ddots\cr
    \hfil\scriptstyle1&\ddots&x_0&x_{1}&x_{2}&\ddots\cr
    \hfil\scriptstyle0&\ddots&x_{-1}&x_0&x_{1}&\ddots\cr
    \hfil\scriptstyle-1&\ddots&x_{-2}&x_{-1}&x_0&\ddots\cr
    \hfil\scriptstyle\vdots&\ddots&\ddots&\ddots&\ddots&\ddots\cr}
\end{equation*}
with support in $\HLa=\{(r,c):r-c\in\La\}$. 

This is a point of departure for the interplay of harmonic analysis and
operator theory. In the general case of a discrete group~$\Ga$, the
counterpart to a bounded measurable function is defined as a bounded
operator on~$\ell^2_\Ga$ whose matrix has the form
$(x_{rc^{-1}})_{(r,c)\in\Ga\times\Ga}$ for some sequence
$(x_\ga)_{\ga\in\Ga}$. This will be the framework of the body of this
article, while the introduction sticks to the case $\Ga=\Z$.

We are concerned with two kinds of multipliers. A sequence
$\ph=(\ph_k)_{k\in\La}$ defines
\begin{itemize}
\item the relative Fourier multiplication operator on trigonometric
  polynomials with spectrum in~$\La$ by
  \begin{equation}
    \sum_{k\in\La}x_kz^k\mapsto\sum_{k\in\La}\ph_kx_kz^k;\label{rfmo}
  \end{equation}
\item the relative Schur multiplication operator on finite matrices
  with support in~$\HLa$ by
  \begin{equation}
    (x_{r,c})_{(r,c)\in\Z\times\Z}\mapsto(\Hph_{r,c}x_{r,c})_{(r,c)\in\Z\times\Z},\label{rsmo}
  \end{equation}
  where $\Hph_{r,c}=\ph_{r-c}$.  
\end{itemize}
Marek Bo\.zejko and Gero Fendler proved that these two multipliers
have the same norm. The operator~\eqref{rfmo} is nothing but the
restriction of~\eqref{rsmo} to Toeplitz matrices. They noted that it
is automatically \emph{completely} bounded: it has the same norm when
acting on trigonometric series with operator coefficients~$x_k$, and
this permits to remove this restriction. Schur multiplication is also
automatically completely bounded.

A part of this observation has been extended by Gilles Pisier to multipliers
acting on a translation invariant Lebesgue space~$\Ell^p_\La$ and on
the subspace~$\Sch^p_{\mkern-3mu\HLa}$ of elements of a
Schatten-von-Neumann class supported by~$\HLa$, respectively; it
yields that the complete norm of a relative Schur
multiplier~\eqref{rsmo} remains bounded by the complete norm of the
relative Fourier multiplier~\eqref{rfmo}.

But $\Ell^p_\La$ is not a subspace of~$\Sch^p_{\mkern-3mu\HLa}$, so
a relative Fourier multiplier may not be viewed anymore as the
restriction of a relative Schur multiplier to Toeplitz matrices. We
point out that this difficulty may be overcome by using Szeg\H{o}'s
limit theorem: a bounded measurable real function on~$\T$ is the
$\textrm{weak}^*$ limit of the normalised counting measure of
eigenvalues of finite truncates of its Toeplitz matrix. This method
also applies to Orlicz norms.
\begin{thm}
   Let\/ $\psi\colon\R^+\to\R^+$ be a continuous nondecreasing function
  vanishing only at~0. The norm of the relative Fourier multiplication operator~\eqref{rfmo} on
    the Lebesgue-Orlicz space\/~$\Ell^\psi_\La$ is bounded by the norm
    of the relative Schur multiplication operator~\eqref{rsmo} on the
    Schatten-von-Neu\-mann-Orlicz class\/~$\Sch^\psi_{\mkern-3mu\HLa}$.
\end{thm}

In order to deal with complete norms, we deduce  a block matrix variant of
Szeg\H{o}'s limit theorem in the style of Erik B\'edos (\cite{be97}),
Theorem~\ref{sz}. Note that
other types of approximation are also available, as the completely
positive approximation property and Reiter sequences combined with
complex interpolation. They are studied in Section~\ref{sec:loc} in
terms of local embeddings of~$\Ell^p$ into~$\Sch^p$. They are more
canonical than Szeg\H{o}'s limit theorem, but give no access to
Orlicz norms.
\begin{thm} 
  Let\/ $\psi\colon\R^+\to\R^+$ be a continuous nondecreasing function
  vanishing only at\/~0. The norm of the following operators is equal:
  \begin{itemize}
  \item the relative Fourier multiplication operator~\eqref{rfmo} on
    the Lebesgue-Orlicz space\/ $\Ell^\psi_\La(\Sch^\psi)$ of\/
    $\Sch^\psi$-valued trigonometric series with spectrum in\/~$\La$;
  \item the relative Schur multiplication operator~\eqref{rsmo} on the
    Schatten-von-Neu\-mann-Orlicz class\/ $\Sch^\psi_{\mkern-3mu\HLa}(\Sch^\psi)$ of\/
    $\Sch^\psi$-valued matrices with support in\/~$\HLa$.
  \end{itemize}
\end{thm}
See Theorems \ref{trprp}~and~\ref{trans2} for the precise statement in the general
case of an amenable group~$\Ga$.

An application of this theorem to the class of all unimodular Fourier
multipliers yields a transfer of lacunary subsets into lacunary matrix
patterns. Call~$\La$ \emph{unconditional in~$\Ell^p$} if
$(z^k)_{k\in\La}$ is an unconditional basis of~$\Ell^p_\La$, and
call~$\HLa$ \emph{unconditional in~$\Sch^p$} if the sequence
$(\e_q)_{q\in\HLa}$ of elementary matrices is an unconditional basis
of~$\Sch^p_{\mkern-3mu\HLa}$. These properties are also known as
$\Lambda(p)$ if $p>2$ ($\Lambda(2)$ if $p<2$) and $\sigma(p)$,
respectively; they have natural ``complete'' counterparts that are
also known as $\Lambda(p)_\cb$ if $p>2$ ($\mathrm{K}(p)_\cb$ if
$p\le2$) and $\sigma(p)_\cb$, respectively. (See Definitions
\ref{unclp:def}~and~\ref{uncsp:def}).
\begin{cor}
  Let $1\le p<\infty$. If\/ $\HLa$ is unconditional in\/~$\Sch^p$, then\/ $\La$ is
  unconditional in\/~$\Ell^p$. \ $\HLa$ is completely unconditional
  in\/~$\Sch^p$ if and only if\/ $\La$ is completely unconditional
  in\/~$\Ell^p$.
\end{cor}
See Proposition~\ref{trans:set} for the precise statement in the
general case of a discrete group~$\Ga$.

The two most prominent multipliers are the Riesz projection
and the Hilbert transform. The first consists in letting $\ph$~be the
indicator function of nonnegative integers and transfers into the
upper triangular truncation of matrices. The second corresponds to the
sign function and transfers into the Hilbert matrix transform. We
obtain the following partial results.

\begin{thm}\label{rhtrht}
  The norm of the matrix Riesz projection and of the matrix
  Hilbert transform on\/ $\Sch^\psi(\Sch^\psi)$ coincide with their norm on\/~$\Sch^\psi$.
  \begin{itemize}
  \item If\/ $p$~is a power of\/~2, then the norm of the matrix Hilbert
    transform on\/~$\Sch^p$ is\/ $\cot(\pi/2p)$.
  \item The norm of the matrix Riesz projection on\/~$\Sch^4$ is\/~$\sqrt2$.
  \end{itemize}
\end{thm}

The transfer technique lends itself naturally to the case where $\La$
contains a sumset $\Row+\Col$: if subsets $\Row'$~and~$\Col'$ are
extracted so that the $r+c$ with $r\in\Row'$ and $c\in\Col'$ are
pairwise distinct, they may play the role of rows and columns.  Here
are the consequences of the conditionality of the sequence of elementary
matrices $\e_{r,c}$ in~$\Sch^p$ for $p\ne2$ and of the unboundedness
of the Riesz transform on $\Sch^1$ and $\Sch^\infty$, respectively.

\begin{thm}If\/ $(z^k)_{k\in\La}$ is a completely unconditional basis of\/
  $\Ell^p_\La$ with\/ $p\ne2$, then\/ $\La$ does not contain sumsets\/
  $\Row+\Col$ of arbitrarily large sets. If either
  \begin{itemize}
  \item the space\/ $\Ell^1_\La$ admits some completely unconditional
    approximating sequence, or
  \item the space\/ $\Cont_\La$ of continuous functions with spectrum
    in\/ $\La$ admits some unconditional approximating sequence,
  \end{itemize}    
  then\/ $\La$ does not contain
  the sumset\/ $\Row+\Col$ of two infinite sets.
\end{thm}

The proof of the second part of this theorem consists in constructing
infinite subsets $\Row'$~and~$\Col'$ and skipped block
sums $\sum(T_{k_{j+1}}-T_{k_j})$ of a given approximating sequence that
act like the projection on the ``upper triangular'' part
of~$\Row'+\Col'$.  See Proposition~\ref{har:lp} and
Theorem~\ref{T} for the precise statement in the general case of a
discrete group $\Ga$.

In the case of quasi-normed Schatten-von-Neumann classes $\Sch^p$ with
$p<1$, the transfer technique yields a new proof for the following
result of Alexey Alexandrov and Vladimir Peller.

\begin{thm}
  Let\/~$0<p<1$. The Fourier multiplier\/ $\ph$ is contractive on\/~$\Ell^p$
  or on\/~$\Ell^p(\Sch^p)$ if and only if the Schur multiplier\/~$\Hph$ is
  contractive on\/~$\Sch^p$ or on\/~$\Sch^p(\Sch^p)$ if and only if the
  sequence\/~$\ph$ is the Fourier transform of an atomic measure of the
  form\/ $\sum a_g \delta_{g}$ on\/~$\T$ with\/~$\sum\abs{a_g}^p\le 1$.
\end{thm}

The emphasis put on \emph{relative} Schur multipliers motivates the
natural question of how the norm of an elementary Schur multiplier,
that is, a rank~1 matrix $(\rh_{r,c})=(x_ry_c)$, gets affected when the
action of~$\rh$ is restricted to matrices with a given support. The
surprising answer is the following theorem.
\begin{thm}
  \label{imps}
  Let\/ $I\subseteq\Row\times\Col$ and consider\/
  $(x_\row)_{\row\in\Row}$~and\/~$(y_\col)_{\col\in\Col}$. The relative
  Schur multiplier on\/~$\Sch^\infty_I$ given by\/ $(x_\row
  y_\col)_{(\row,\col)\in I}$ has norm\/~$\sup_{(\row,\col)\in
    I}\abs{x_\row y_\col}$.
\end{thm}

Let us finally describe the content of this
article. Section~\ref{sec:tr} develops transfer techniques for
Fourier and Schur multipliers provided by a block matrix Szeg\H{o}
limit theorem. This theorem provides local embeddings of $\Ell^\psi$
into $\Sch^\psi$. Section~\ref{sec:loc} shows how interpolation may be
used to define such embeddings for the scale of $\Ell^p$
spaces. Section~\ref{sec:lac} is devoted to the transfer of lacunary
sets into lacunary matrix patterns; the unconditional constant of a
set $\La$ is related to the size of the sumsets it
contains. Section~\ref{sec:tsp1} deals with Toeplitz Schur multipliers
for $p<1$ and comments on the case $p\ge1$. The Riesz projection and
the Hilbert transform are studied in Section~\ref{sec:rpht}. In
Section~\ref{sec:uap}, the presence of sumsets in a spectrum~$\La$ is
shown to be an obstruction for the existence of completely
unconditional bases for $\Ell^p_\La$. The last section provides a
norm-preserving extension for partially specified rank~1 Schur
multipliers.
 
\begin{NotTer}
  Let~$\T=\{z\in\C:|z|=1\}$ be the circle. 
  
  Given an index set~$\Col$ and~$c\in \Col$, \ $\e_c$~is the sequence
  defined on~$\Col$ as the indicator function~$\un_{\{c\}}$ of the
  singleton~$\{c\}$, so that $(\e_c)_{c\in \Col}$~is the canonical
  Schauder basis of the Hilbert space of square summable sequences
  indexed by~$\Col$, denoted by~$\ell^2_\Col$. We will use the
  notation~$\ell^2_n=\ell^2_{\{1,2,\dots,n\}}$
  and~$\ell^2=\ell^2_{\N}$.

  Given a product set~$I=\Row\times\Col$ and~$q=(r,c)$, the indicator
  function~$\e_q=\e_{r,c}$ is the \emph{elementary matrix} identified
  with the linear operator from~$\ell^2_\Col$ to~$\ell^2_\Row$ that
  maps~$\e_\col$ to~$\e_\row$ and all other basis vectors to~0. The
  \emph{matrix coefficient} at coordinate~$q$ of a linear operator~$x$
  from~$\ell^2_\Col$ to~$\ell^2_\Row$ is~$x_q=\tr\e_q^*x$, and its
  \emph{matrix representation}
  is~$(x_q)_{q\in\Row\times\Col}=\sum_{q\in\Row\times \Col}x_q\e_q$.
  The \emph{support} or \emph{pattern} of~$x$
  is~$\{q\in\Row\times\Col:x_q\ne0\}$.

  The space of all bounded operators from~$\ell^2_\Col$
  to~$\ell^2_\Row$ is denoted by~$\borne(\ell^2_\Col,\ell^2_\Row)$,
  and its subspace of compact operators is denoted by~$\Sch^\infty$.

  Let $\psi\colon\R^+\to\R^+$ be a continuous nondecreasing function
  vanishing only at~0. The \emph{Schatten-von-Neumann-Orlicz
    class}~$\Sch^\psi$ is the space of those compact operators~$x$
  from~$\ell^2_\Col$ to~$\ell^2_\Row$ such that
  $\tr\psi(|x|/a)<\infty$ for some $a>0$. If $\psi$~is convex, then
  $\Sch^\psi$~is a Banach space for the norm given by
  $\|x\|_{\Sch^\psi}=\inf\{a>0:\tr\psi(|x|/a)\le1\}$. Otherwise,
  $\Sch^\psi$ is a Fr\'echet space for the F-norm given by
  $\|x\|_{\Sch^\psi}=\inf\{a>0:\tr\psi(|x|/a)\le a\}$ (see
  \cite[Chapter~3]{or92}). This space may also be constructed as the
  noncommutative Lebesgue-Orlicz space~$\Ell^\psi(\tr)$ associated with
  a corner of the von~Neumann algebra~$\borne(\ell^2_\Col\oplus 
\ell^2_\Row)$ endowed
  with the normal, faithful, semifinite trace~$\tr$.  If $\psi$~is the
  power function~$t\mapsto t^p$, this space is denoted~$\Sch^p$; if
  $p\ge1$, then $\|x\|_{\Sch^p}={(\tr|x|^p)}^{1/p}$; if $p<1$, then
  $\|x\|_{\Sch^p}={(\tr|x|^p)}^{1/(1+p)}$.

  If~$\card\Col=\card\Row=n$, then $\borne(\ell^2_\Col,\ell^2_\Row)$
  identifies with the space of $n\times n$ matrices denoted~$\mat_n$,
  and we write~$\Sch^\psi_n$ for~$\Sch^\psi$. Let
  $(\Row_n\times\Col_n)$~be a sequence of finite sets such that each
  element of~$\Row\times\Col$ eventually is
  in~$\Row_n\times\Col_n$. Then the sequence of operators~$P_n\colon
  x\mapsto\sum_{q\in\Row_n\times \Col_n}x_q\e_q$ tends pointwise to
  the identity on~$\Sch^\psi$.
  
  For~$I\subseteq\Row\times\Col$, we define the space~$\Sch^\psi_I$ as
  the closed subspace of~$\Sch^\psi$ spanned by~$(\e_q)_{q\in I}$; 
  this coincides with the subspace of those~$x\in \Sch^\psi$ whose
  support is a subset of~$I$.
  
  A \emph{relative Schur multiplier} on~$\Sch^\psi_I$ is a
  sequence~$\rh=(\rh_q)_{q\in I}\in\C^I$ such that the associated
  \emph{Schur multiplication operator}~$\Schur_\rh$ defined
  by~$\e_q\mapsto\rh_q\e_q$ for~$q\in I$ is bounded on~$\Sch^\psi_I$.
  The norm~$\|\rh\|_{\Schur(\Sch^\psi_I)}$ of~$\rh$ is defined as the
  norm of~$\Schur_\rh$.  This norm is the supremum of the norm of its
  restrictions to finite rectangle sets~$\Row'\times\Col'$.  We used
  \cite{pi98,pi01} as a reference.
  
  Let $\Ga$~be a discrete group with identity~$\id$. The \emph{reduced
    $\Cont^*$-algebra} of~$\Ga$ is the closed subspace spanned by the
  left translations~$\la_\ga$ (the linear operators defined
  on~$\ell^2_\Ga$ by~$\la_\ga\e_\beta=\e_{\ga\beta}$)
  in~$\borne(\ell^2_\Ga)$; we denote it by~$\Cont$, set in roman
  type. The \emph{von~Neumann algebra} of~$\Ga$ is its
  $\textrm{weak}^*$ closure, endowed with the normal, faithful,
  normalised finite trace~$\tau$ defined by~$\tau(x)=x_{\id,\id}$; we
  denote it by~$\vn$.  Let $\psi\colon\R^+\to\R^+$ be a continuous
  nondecreasing function vanishing only at~0. We define the
  noncommutative Lebesgue-Orlicz space~$\Ell^\psi$ of $\Ga$ as the
  completion of $\vn$ with respect to the norm given by
  $\|x\|_{\Ell^\psi}=\inf\{a>0:\tau(\psi(|x|/a))\le1\}$ if $\psi$~is
  convex, and with respect to the F-norm given by
  $\|x\|_{\Ell^\psi}=\inf\{a>0:\tau(\psi(|x|/a))\le a\}$ otherwise.
  If $\psi$~is the power function~$t\mapsto t^p$, this space is
  denoted~$\Ell^p$; if $p\ge1$, then
  $\|x\|_{\Ell^p}={\tau(|x|^p)}^{1/p}$; if $p<1$, then
  $\|x\|_{\Ell^p}={\tau(|x|^p)}^{1/(1+p)}$.  The \emph{Fourier
    coefficient} of~$x$ at~$\ga$
  is~$x_\ga=\tau(\la_\ga^*x)=x_{\ga,\id}$ and its \emph{Fourier
    series} is~$\sum_{\ga\in\Ga}x_\ga\la_\ga$. The \emph{spectrum} of
  an element~$x$ is~$\{\ga\in\Ga:x_\ga\ne0\}$. Let $X$~be the
  $\Cont^*$-algebra~$\Cont$ or the space~$\Ell^\psi$ and
  let~$\La\subseteq\Ga$; then we define $X_\La$~as the closed subspace
  of~$X$ spanned by the~$\la_\ga$ with~$\ga\in\La$. We skip the
  general question of when this coincides with the subspace of
  those~$x\in X$ whose spectrum is a subset of~$\La$, but note that
  this is the case if $\Ga$ is an amenable group (or if $\Ga$ has the
  AP and $\vn$~has the QWEP by \cite[Theorem~4.4]{jr03}) and $\psi$~is
  the power function~$t\mapsto t^p$. Note also that our definition of
  $X_\La$ makes it a subspace of the \emph{heart} of $X$: if $x\in
  X_\La$, then $\tau(\psi(|x|/a))$ is finite for all $a>0$.

  A \emph{relative Fourier multiplier} on~$X_\La$ is a
  sequence~$\ph=(\ph_\ga)_{\ga\in\La}\in\C^\La$ such that the
  associated Fourier multiplication operator~$\Fourier_\ph$ defined
  by~$\la_\ga\mapsto\ph_\ga\la_\ga$ for~$\ga\in\La$ is bounded
  on~$X_\La$. The norm~$\|\ph\|_{\Fourier(X_\La)}$ of~$\ph$ is defined
  as the norm of~$\Fourier_\ph$. 
  Fourier multipliers on the whole of the $\Cont^*$-algebra $\Cont$
  are also called \emph{multipliers of the Fourier algebra ${\mathrm
      A}(\Ga)$} (which may be identified with $\Ell^1$); they form the
  set $\Fourier(\mathrm{A}(\Ga))$.

  The space~$\Sch^\psi(\Sch^\psi)$ is the space of those compact
  operators~$x$ from~$\ell^2\otimes\ell^2_\Col$
  to~$\ell^2\otimes\ell^2_\Row$ such that
  $\|x\|_{\Sch^\psi(\Sch^\psi)}=\inf\{a:\tr\otimes\tr\psi(|x|/a)\le1\}$:
  it is the noncommutative Lebesgue-Orlicz
  space~$\Ell^\psi(\tr\otimes\tr)$ associated with a corner of the von~Neumann
  algebra~$\borne(\ell^2)\otimes\borne(\ell^2_\Col\oplus \ell^2_R)$. One may
  think of~$\Sch^\psi(\Sch^\psi)$ as the $\Sch^\psi$-valued
  Schatten-von-Neumann class; we define the matrix coefficient of~$x$
  at~$q$ by~$x_q = (\Id_{\Sch^\psi}\otimes\tr)\allowbreak
  \bigl((\Id_{\ell^2}\otimes\e_q^*)x\bigr)\in\Sch^\psi$ and its matrix
  representation by $\sum_{q\in\Row\times\Col}x_q\tens\e_q$. The
  support of~$x$ and the subspace~$\Sch^\psi_I(\Sch^\psi)$ are defined
  in the same way as~$\Sch^\psi_I$.
  
  Similarly, the space~$\Ell^\psi(\tr\otimes\tau)$ is the noncommutative
  Lebesgue-Orlicz space associated with the von~Neumann algebra~$\borne(\ell^2)\otimes\vn=\Ell^\infty(\tr\otimes\tau)$. One may think of~$\Ell^\psi(\tr\otimes\tau)$ as the $\Sch^\psi$-valued noncommutative
  Lebesgue space; we define the Fourier coefficient of~$x$ at~$\ga$
  by~$x_\ga=(\Id_{\Sch^\psi}\otimes\tau)\bigl((\Id_{\ell^2}\otimes\la_\ga^*)x\bigr)\in\Sch^\psi$
  and its Fourier series
  by~$\sum_{\ga\in\Ga}x_\ga\tens\la_\ga$; the spectrum
  of~$x$ is defined accordingly. The
  subspace~$\Ell^\psi_\La(\tr\otimes\tau)$ is the closed subspace
  of~$\Ell^\psi(\tr\otimes\tau)$ spanned by
  the~$x\otimes\la_\ga$ with~$x\in\Sch^\psi$
  and~$\ga\in\La$. 
  
  An operator~$T$ on~$\Sch^\psi_I$ is \emph{bounded on
    $\Sch^\psi_I(\Sch^\psi)$} if the linear operator $\Id_{\Sch^\psi}\otimes T$
  defined by~$x\tens y\mapsto x\tens T(y)$ for~$x\in \Sch^\psi$ and $y$
  in~$\Sch^\psi_I$ on finite tensors extends to a bounded
  operator~$\Id_{\Sch^\psi}\otimes T$ on~$\Sch^\psi_I(\Sch^\psi)$.  The norm of
  a Schur multiplier~$\rh$ on $\Sch^\psi_I(\Sch^\psi)$ is defined as the
  norm of~$\Id_{\Sch^\psi}\otimes\Schur_\rh$.
  Similar definitions hold for an operator~$T$ on~$\Ell^\psi_\La$; the
  norm of a Fourier multiplier~$\ph$ on~$\Ell^\psi_\La(\tr\otimes\tau)$
  is the norm of $\Id_\Sch^\psi\otimes\Fourier_\ph$
  on~$\Ell^\psi_\La(\tr\otimes\tau)$.

  Let $\psi$~be the power function~$t\mapsto t^p$ with $p\ge1$; the
  norms on $\Sch^p(\Sch^p)$ and $\Ell^p(\tr\otimes\tau)$ describe the
  canonical operator space structure on $\Sch^p$~and~$\Ell^p$,
  respectively (see \cite[Corollary~1.4]{pi98}); we should rather use
  the notation $\Sch^p[\Sch^p]$ and $\Sch^p[\Ell^p]$. This explains
  the following terminology.  An operator~$T$ on~$\Sch^p_I$ is
  \emph{completely bounded} (c.b.) if $\Id_{\Sch^p}\otimes
  T$~is bounded on~$\Sch^p_I(\Sch^p)$; the norm
  of~$\Id_{\Sch^p}\otimes T$ is the \emph{complete norm} of~$T$
  (compare \cite[Lemma~1.7]{pi98}).  The complete
  norm~$\|\rh\|_{\Schur_\cb(\Sch^p_I)}$ of a Schur multiplier~$\rh$ is
  defined as the complete norm of~$\Schur_\rh$. Note that the complete
  norm of a Schur multiplier~$\rh$ on~$\Sch^\infty_I$ is equal to its
  norm (\cite[Theorem~3.2]{pps89}): \
  $\|\rh\|_{\Schur_\cb(\Sch^\infty_I)}=\|\rh\|_{\Schur(\Sch^\infty_I)}$.
  The complete norm~$\|\ph\|_{\Fourier_\cb(\Ell^p_\La)}$ of a Fourier
  multiplier~$\ph$ is defined as the complete norm of~$\Fourier_\ph$.
  The complete norm of an operator~$T$ on~$\Cont_\La$ is the norm
  of~$\Id_{\Sch^\infty}\otimes T$ on the subspace
  of~$\Sch^\infty\otimes\Cont$ spanned by the~$x\otimes\la_\ga$
  with~$x\in\Sch^\infty$ and~$\ga\in\La$. In the case $\La=\Ga$, $\ph$
  is also called a \emph{c.b.\ multiplier of the Fourier algebra
    ${\mathrm A}(\Ga)$} and one writes $\ph\in\Fourier_{\cb}(\mathrm
  A(\Ga))$. If $\Ga$~is amenable, the complete norm of a Fourier
  multiplier~$\ph$ on~$\Cont_\La$ is equal to its norm: \
  $\|\ph\|_{\Fourier_\cb(\Cont_\La)}=\|\ph\|_{\Fourier(\Cont_\La)}$
  (this follows from \cite[Corollary~1.8]{dh85} as shown by the proof of Theorem~\ref{trans2}$\,(c)$).

  An element whose norm is at most~1 is
  \emph{contractive}, and if its complete norm is at most~1, it is
  \emph{completely} contractive.

  If $\Ga$~is abelian, let $G$~be its dual group and endow it with its
  unique normalised Haar measure~$m$. Then the Fourier transform
  identifies the $\Cont^*$-algebra~$\Cont$ as the space of continuous functions on~$G$, \
  $\Ell^\infty$ as the space of classes of bounded measurable
  functions on~$(G,m)$, \ $\Ell^\psi$ as the Lebesgue-Orlicz space of classes
  of~$\psi$-integrable functions on~$(G,m)$, \ $\tau(x)$
  as~$\int_Gx(g)\,\mathrm dm(g)$, \ $\Ell^\psi(\tr\otimes\tau)$ as the
  $\Sch^\psi$-valued Lebesgue-Orlicz space $\Ell^\psi(\Sch^\psi)$  and
  $x_\ga$ as~$\hat x(\ga)$. 

\end{NotTer}

\section{Transfer between Fourier and Schur multipliers}\label{sec:tr}

Let $\La$~be a subset of a discrete group~$\Ga$ and let $\ph$~be a
relative Fourier multiplier on~$\Cont_\La$, the closed subspace spanned by $(\la_\ga)_{\ga\in\La}$ in the reduced $\Cont^*$-algebra of $\Ga$. Let $x\in\Cont_\La$; the
matrix of~$x$ is \emph{constant down the diagonals} in the sense that
for every~$(r,c)\in\Ga\times\Ga$, \
$x_{r,c}=x_{rc^{-1},\id}=x_{rc^{-1}}$. We say that $x$~is a
\emph{Toeplitz} operator on~$\ell^2_\Ga$. Furthermore, the matrix of
the Fourier product~$\Fourier_\ph x$ of~$\ph$ with~$x$ is given
by~$(\Fourier_\ph x)_{r,c}=\ph_{rc^{-1}}x_{r,c}$. This equality shows that if we
set~$\HLa=\{(r,c)\in\Ga\times\Ga:\row\col^{-1}\in
\La\}$~and~$\Hph_{r,c}=\ph_{\row\col^{-1}}$, then $\Fourier_\ph x$~is the Schur
product $\Schur_\Hph x$ of~$\Hph$ with~$x$.  We have transferred the
Fourier multiplier~$\ph$ into the Schur multiplier~$\Hph$. This proves
at once that the norm of the Fourier multiplier~$\ph$ on~$\Cont_\La$
is the norm of the Schur multiplier~$\Hph$ on the subspace of Toeplitz
elements of~$\borne(\ell^2_\Ga)$ with support in~$\HLa$, and that the
same holds for complete norms.

We shall now provide the means to generalise this identification to
the setting of Lebesgue-Orlicz spaces $\Ell^\psi$. We shall bypass the main obstacle, that $\Ell^\psi$~may not
be considered as a subspace of~$\Sch^\psi$, by the Szeg\H{o} limit
theorem as stated by Erik B\'edos (\cite{be97}).

Consider a discrete amenable group~$\Ga$; it admits a
\emph{F{\o}lner averaging net of sets} $(\Ga_\iota)$, that is,
\begin{itemize}
\item each~$\Ga_\iota$ is a finite subset of~$\Ga$;
\item
  $\card(\ga\Ga_\iota\Delta\Ga_\iota)=o(\card\Ga_\iota)$
  for each~$\ga\in\Ga$.
\end{itemize}
Each set~$\Ga_\iota$ corresponds to the orthogonal
projection~$p_\iota$ of~$\ell^2_\Ga$ onto its
$(\card\Ga_\iota)$-di\-men\-sio\-nal subspace of sequences supported
by~$\Ga_\iota$. The \emph{truncate} of a selfadjoint
operator~$y\in\borne(\ell^2_\Ga)$ with respect to~$\Ga_\iota$ is
$y_\iota=p_\iota^{\vphantom{*}} yp_\iota^*$; it has $\card\Ga_\iota$
eigenvalues~$\alpha_j$, counted with multiplicities, and its
\emph{normalised counting measure of eigenvalues} is
\begin{equation*}
  \mu_\iota=\frac1{\card\Ga_\iota}\sum_{j=1}^{\card\Ga_\iota}\delta_{\alpha_j}.
\end{equation*}
If $y$~is a Toeplitz operator, that is, if $y\in\vn$, Erik B\'edos
(\cite[Theorem~10]{be97}) proved that $(\mu_\iota)$ converges
$\textrm{weak}^*$ to the \emph{spectral measure} of~$y$ with respect
to~$\tau$, which is the unique Borel probability measure~$\mu$ on~$\R$
such that
\begin{equation*}
  \tau(f(y))=\int_\R f(\alpha)\mathrm d\mu(\alpha)
\end{equation*}
for every continuous function~$f$ on~$\R$ that tends to zero at
infinity. If $\Ga$~is abelian, then $y$~may be identified as the
class of a real-valued bounded measurable function on the group~$G$
dual to~$\Ga$ and $\mu$~is the distribution of~$y$.

Let us now state and prove the $\Ell^\psi$ version of the identification
described at the beginning of this section. 
\begin{thm}
  \label{trprp}
  Let\/ $\Ga$~be a discrete amenable group and let\/ $\psi\colon\R^+\to\R^+$
  be a continuous nondecreasing function vanishing only at~0. Let\/
  $\La\subseteq\Ga$ and\/ $\ph\in\C^\La$. Consider the associated
  Toeplitz set\/ $\HLa=\{(r,c)\in\Ga\times\Ga:\row\col^{-1}\in\La\}$
  and the Toeplitz matrix defined
  by\/ $\Hph_{r,c}=\ph_{\row\col^{-1}}$. The norm of the relative Fourier multiplier\/~$\ph$
    on\/~$\Ell^\psi_\La$ is bounded by the norm of the relative
    Schur multiplier\/~$\Hph$ on\/~$\Sch^\psi_{\mkern-3mu\HLa}$.
 \end{thm}
\begin{proof}
  A Toeplitz matrix has the form $(x_{rc^{-1}})_{(r,c)\in \HLa}$. Our
  definition of the space~$\Ell^\psi_\La$ (in the section on Notation and terminology) ensures that
  we may suppose that only a finite number of the~$x_\ga$ are nonzero
  for the computation of the norm of~$\ph$. Then
  $(x_{rc^{-1}})_{(r,c)\in \HLa}$ is the matrix of the
  operator $x=\sum_{\ga\in \La}x_\ga\la_\ga$ for the canonical basis
  of~$\ell^2_\Ga$.

  Let~$y=x^*x$ and let
  $\tilde\psi$~be a continuous function with compact support such that
  $\tilde\psi(t)=\psi(t)$ on $[0,\|y\|]$. By Szeg\H{o}'s limit theorem,
  \begin{equation*}
\frac1{\#\Ga_\iota}\tr\psi(y_\iota)=
\frac1{\#\Ga_\iota}\tr\tilde\psi(y_\iota)
    \to\tau(\tilde\psi(y))=\tau(\psi(y)).
  \end{equation*}
  We have $y_\iota={(xp_\iota^*)}^*(xp_\iota^*)$; let us describe how $\Hph$ acts on~$xp_\iota^*$. Schur multiplication
  with~$\Hph$ transforms the matrix of~$xp_\iota^*$, that is, the
  truncated Toeplitz matrix~$(x_{rc^{-1}})_{(r,c)\in
    \HLa\cap\Ga\times\Ga_\iota}$, into the matrix~$(\ph_{rc^{-1}}x_{rc^{-1}})_{(r,c)\in
    \HLa\cap\Ga\times\Ga_\iota}$ so that it transforms~$xp_\iota^*$
  into~$(\Fourier_\ph x)p_\iota^*$.
\end{proof}

\begin{rem}
  In the case of a finite abelian group, no limit theorem is needed.
  This case was considered
  in~\cite[Proposition~2.5$\,(b)$]{ne06}; compare with \cite[Chapter
  6, Lemma 3.8]{pe03}.
\end{rem}

\begin{rem}\label{trrem}
  Our technique proves in fact that the norm of a Fourier multiplier
  is the upper limit of the norm of the corresponding relative Schur multipliers on
  subspaces of truncated Toeplitz matrices. We ignore whether or not it is
  actually their supremum. 
\end{rem}

Remark~\ref{b-cb} illustrates that the two norms in
Theorem~\ref{trprp} are different in general. This is not so in the $\Sch^\psi$-valued case because of the following argument.
It has been used (first in~\cite{bf84},
see~\cite[Proposition~D.6]{bo08}) to show that the complete norm of the
Fourier multiplier~$\ph$ on~$\Cont_\La$ bounds the complete norm of
the Schur multiplier~$\Hph$ on $\Sch^\infty_{\mkern-3mu\HLa}$, so that we have in full
generality
$\|\ph\|_{\Fourier_\cb(\Cont_\La)}=\|\Hph\|_{\Schur_\cb(\Sch^\infty_{\mkern-3mu\HLa})}$.
\begin{lem}
  \label{trans1}
  Let\/ $\Ga$~be a discrete group and let\/ $\Row$~and\/~$\Col$ be subsets
  of\/~$\Ga$. With\/~$\La\subseteq\Ga$
  associate\/~$\HLa=\{(r,c)\in\Row\times\Col:\row\col^{-1}\in \La\}$;
  given\/~$\ph\in\C^\La$, define\/~$\Hph\in\C^\HLa$
  by\/~$\Hph_{r,c}=\ph_{\row\col^{-1}}$.  Let\/ $\psi\colon\R^+\to\R^+$ be
  a continuous nondecreasing function vanishing only at\/~0. The norm of
  the relative Schur multiplier\/~${\Hph}$
  on\/~$\Sch^\psi_{\mkern-3mu\HLa}(\Sch^\psi)$ is bounded by the norm of
  the relative Fourier multiplier\/~$\ph$
  on\/~$\Ell^\psi_\La(\tr\otimes\tau)$.
\end{lem}
\begin{proof}
  We adapt the argument in~\cite[Lemma~8.1.4]{pi98}. Let~$x_q\in
  \Sch^\psi$, of which only a finite number are nonzero. The space~$\Ell^\psi(\tr\otimes\tr\otimes\tau)$ is a left and right
  $\Ell^\infty(\tr\otimes\tr\otimes\tau)$-module,
  and~$\sum_{\ga\in\Ga}\e_{\ga\ga}\tens\la_\ga$
  is a unitary in~$\Ell^\infty(\tr\otimes\tau)$ so that
  \begin{multline*}
     \Bigl\|\sum_{q\in
        \HLa}x_q\tens\e_q\Bigr\|_{\Sch^\psi_{\mkern-3mu\HLa}(\Sch^\psi)}\\
      \begin{aligned}
      &=\Bigl\|\Bigl(\Id\otimes\sum_{r\in\Row}\e_{r,r}\tens\la_\row\Bigr)
      \Bigl(\sum_{q\in \HLa}x_q\tens\e_q\tens\la_\id\Bigr)
      \Bigl(\Id\otimes\sum_{c\in\Col}\e_{c,c}\tens\la_\col^*\Bigr)\Bigr\|_{\Ell^\psi(\tr\otimes\tr\otimes\tau)}\\
      &= \biggl\| \sum_{(r,c)\in \HLa}
      x_{r,c}\tens\e_{r,c}\tens\la_{\row\col^{-1}}\biggr\|_{\Ell^\psi(\tr\otimes\tr\otimes\tau)}\\
      &= \biggl\| \sum_{\ga\in \La} \biggl(\sum_{\row\col^{-1}=\ga}
      x_{r,c}\tens\e_{r,c}\biggr)\tens\la_\ga\biggr\|_{\Ell^\psi_\La(\tr\otimes\tr\otimes\tau)}.
    \end{aligned}
  \end{multline*}
 This yields an isometric embedding of~$\Sch^\psi_{\mkern-3mu\HLa}(\Sch^\psi)$ in~$\Ell^\psi_\La(\tr\otimes\tr\otimes\tau)$. As $\Sch^\psi(\Sch^\psi)$~is the
  Schatten-von-Neumann-Orlicz class for the Hilbert space~$\ell^2\otimes\ell^2_\Ga$, which may be identified with~$\ell^2$,
  \begin{multline*}
    \Bigl\| \sum_{q\in \HLa} x_q\tens\Hph_q\e_q \Bigr\|_{\Sch^\psi_{\mkern-3mu\HLa}(\Sch^\psi)}
    = \biggl\| \sum_{\ga\in \La} \biggl(
    \sum_{\row\col^{-1}=\ga} x_{r,c}\tens\e_{r,c}
    \biggr)\tens\ph_\ga     \la_\ga \biggr\|_{\Ell^\psi_\La(\tr\otimes\tr\otimes\tau)}\\
    \le \|\Id_{\Sch^\psi}\otimes\Fourier_\ph\| \Bigl\| \sum_{q\in \HLa} x_q\tens\e_q
    \Bigr\|_{\Sch^\psi_{\mkern-3mu\HLa}(\Sch^\psi)}.\tag*{\qedhere}
  \end{multline*}
\end{proof}

\begin{rem}
  \label{transg}
  This proof also shows the following transfer: let $(r_i)$ and
  $(c_j)$ be sequences in~$\Ga$,
  consider~$\bLa=\{(i,j)\in\N\times\N:r_ic_j\in\La\}$ and
  define~$\bph\in\C^\bLa$ by~$\bph(i,j)=\ph(r_ic_j)$. Then the norm of
  the relative Schur multiplier~${\bph}$
  on~$\Sch^\psi_{\bLa}(\Sch^\psi)$ is bounded by the norm of the
  relative Fourier multiplier~$\Id_{\Sch^\psi}\otimes\Fourier_\ph$
  on~$\Ell^\psi_\La(\tr\otimes\tau)$ (compare with
  \cite[Theorem~6.4]{pi01}). In particular, if the $r_ic_j$ are
  pairwise distinct, this permits us to transfer every Schur multiplier,
  not just the Toeplitz ones. See \cite[Section~11]{ne06} for
  applications of this transfer.
\end{rem}

We shall now prove that the two norms in this lemma are in fact equal.
As we want to compute norms of multipliers on $\Sch^\psi$-valued
spaces, we shall generalise the Szeg\H{o} limit theorem to the block
matrix case, which was not considered in \cite{be97}. This is the analogue of
the scalar case for selfadjoint elements $y\in\mat_n\otimes\vn$, whose
\emph{$\mat_n$-valued spectral measure}~$\mu$ is defined by
\begin{equation*}
  \int_\R f(\alpha)\mathrm d\mu(\alpha)=\Id_{\mat_n}\otimes\tau(f(y))
\end{equation*}
for every continuous function $f$ on $\R$ that tends to zero at infinity.

The orthogonal projection~$\tilde
p_\iota=\Id_{\ell^2_n}\otimes p_\iota$ defines the truncate~$y_\iota=\tilde p_\iota^{\vphantom{*}} y\tilde
p_\iota^*\in\mat_n\otimes\borne(\ell^2_{\Ga_\iota})$, and the \emph{$\mat_n$-valued normalised counting
measure of eigenvalues}~$\mu_\iota$ by
\begin{equation*}
  \int_\R f(\alpha)\mathrm d\mu_\iota(\alpha)=\Id_{\mat_n}\otimes\frac{\tr}{\card\Ga_\iota}(f(y_\iota))
\end{equation*}
for every continuous function $f$ on $\R$ that tends to zero at infinity.

\begin{thm}[Matrix Szeg\H{o} limit theorem]
  \label{sz}
  Let\/ $\Ga$~be a discrete amenable group and let\/ $(\Ga_\iota)$~be a F{\o}lner averaging net for\/~$\Ga$. Let\/ $y$~be a selfadjoint
  element of\/~$\mat_n\otimes\vn$. The net\/~$(\mu_\iota)$ of\/~$\mat_n$-valued normalised
  counting measures of eigenvalues of the truncates of\/~$y$ with respect
  to\/~$\Ga_\iota$ converges in the\/ $\textrm{weak}^*$ topology to
  the spectral measure of\/~$y$:
  \begin{equation*}
    \int_\R f(\alpha)\mathrm d\mu_\iota(\alpha)\to\Id_{\mat_n}\otimes\tau(f(y))
  \end{equation*}
  for every continuous function\/~$f$ on\/~$\R$ that tends to zero at
  infinity.
\end{thm}

\begin{proof}[Sketch of proof]
  We first suppose that $y=\sum_{\ga\in\Ga}y_\ga\tens\la_\ga$ with
  only a finite number of the~$y_\ga\in\mat_n$ nonzero. The
  $\mat_n$-valued matrix of the truncate~$y_\iota$ of~$y$ for the
  canonical basis of~$\ell^2_{\Ga_\iota}$
  is~$(y_{rc^{-1}})_{(r,c)\in\Ga_\iota\times\Ga_\iota}$. As the
  truncates~$y_\iota$ of~$y$ are uniformly bounded, it suffices to
  prove that
  \begin{equation*}
    \Id\otimes\frac\tr{\#\Ga_\iota}(y_\iota^k)\to\Id\otimes\tau(y^k)
  \end{equation*}
  for every~$k$. This is trivial if~$k=0$. If~$k=1$, then
  \begin{equation*}
    \Id\otimes\frac\tr{\#\Ga_\iota}(y_\iota)
    =\frac1{\#\Ga_\iota}\sum_{\col\in\Ga_\iota}y_{\col,\col}
    =\Id\otimes\tau(y)
  \end{equation*}
  as~$y_{\col,\col}=y_{\col\col^{-1}}=y_\id$. If~$k\ge2$, the same
  formula holds with~$y^k$ instead of~$y$:
  \begin{equation*}
    \Id\otimes\tau(y^k)=\Id\otimes\frac\tr{\#\Ga_\iota}
    (\tilde p_\iota^{\vphantom{*}} y^k\tilde p_\iota^*),
  \end{equation*}
  so that we wish to prove
  \begin{equation*}
    \Id\otimes\tr(\tilde p_\iota^{\vphantom{*}} y^k\tilde p_\iota^*
    -{(\tilde p_\iota^{\vphantom{*}} y\tilde p_\iota^*)}^k)
    =o(\card\Ga_\iota).
  \end{equation*}
  Note that 
  \begin{equation*}
    \bignorm{\Id\otimes\tr\bigl(\tilde p_\iota^{\vphantom{*}} y^k\tilde p_\iota^*-{(\tilde p_\iota^{\vphantom{*}} y\tilde p_\iota^*)}^k\bigr)}_{\Sch^1_n}
    \le\|\tilde p_\iota^{\vphantom{*}} y^k\tilde p_\iota^*-{(\tilde p_\iota^{\vphantom{*}} y\tilde p_\iota^*)}^k\|_{\Sch^1(\Sch^1_n)}.
  \end{equation*}
  Lemma~5 in \cite{be97} provides the following estimate. As
  \begin{equation*}
    \tilde p_\iota^{\vphantom{*}} y^k\tilde p_\iota^*-{(\tilde p_\iota^{\vphantom{*}} y\tilde p_\iota^*)}^k
    =\tilde p_\iota^{\vphantom{*}} y^{k-1}(y\tilde p_\iota^*-\tilde p_\iota^*\tilde p_\iota^{\vphantom{*}} y\tilde p_\iota^*)
    +(\tilde p_\iota^{\vphantom{*}} y^{k-1}\tilde p_\iota^*-{(\tilde p_\iota^{\vphantom{*}} y\tilde p_\iota^*)}^{k-1})\tilde p_\iota^{\vphantom{*}}y\tilde p_\iota^*,
  \end{equation*}
  an induction yields 
  \begin{equation*}
    \|\tilde p_\iota^{\vphantom{*}} y^k\tilde p_\iota^*-{(\tilde p_\iota^{\vphantom{*}} y\tilde p_\iota^*)}^k\|_{\Sch^1(\Sch^1_n)}
      \le(k-1)\|y\|_{\mat_n\otimes\vn}^{k-1}\|y\tilde p_\iota^*-\tilde p_\iota^*\tilde p_\iota^{\vphantom{*}} y\tilde p_\iota^*\|_{\Sch^1(\Sch^1_n)}.
  \end{equation*}
  It suffices to consider the very last norm for each term~$y_\ga\tens\la_\ga$ of~$y$: let~$h\in\ell^2_n$ and~$\beta\in\Ga$; as
  \begin{equation*}
    \bigl((y_\ga\tens\la_\ga)\tilde p_\iota^*-\tilde p_\iota^*\tilde p_\iota^{\vphantom{*}}(y_\ga\tens\la_\ga)\tilde p_\iota^*\bigr)(h\otimes\e_\beta)
    =
    \begin{cases}
      y_\ga(h)\e_{\ga\beta}&\text{if $\beta\in\Ga_\iota$ and $\ga\beta\notin\Ga_\iota$}\\
      0&\text{otherwise,}
    \end{cases}
  \end{equation*}
  the definition of a F{\o}lner averaging net yields
  \begin{equation*}
    \|(y_\ga\tens\la_\ga)\tilde p_\iota^*
    -\tilde p_\iota^*\tilde p_\iota^{\vphantom{*}} (y_\ga\tens\la_\ga)\tilde p_\iota^*\|
    _{\Sch^1(\Sch^1_n)}
    \le\#(\Ga_\iota\setminus\ga^{-1}\Ga_\iota)\|y_\ga\|_{\Sch^1_n}
    =o(\card\Ga_\iota).
  \end{equation*}
  An approximation argument as in the proof of \cite[Proposition~4]{be97}
  permits us to conclude for~$y\in\mat_n\otimes\vn$.
\end{proof}

Here is the promised strengthening of Lemma~\ref{trans1} together with three variants.

\begin{thm}\label{trans2}
  Let\/ $\Ga$~be a discrete amenable group. Let\/
  $\La\subseteq\Ga$ and\/ $\ph\in\C^\La$. Consider the
  associated Toeplitz set\/~$\HLa=\{(r,c)\in\Ga\times\Ga:\row\col^{-1}\in
  \La\}$ and the Toeplitz matrix defined by\/~$\Hph_{r,c}=\ph_{\row\col^{-1}}$.
  \begin{enumerate}
  \item Let\/ $\psi\colon\R^+\to\R^+$ be a continuous nondecreasing
    function vanishing only at\/~0. The norm of the relative Fourier
    multiplier\/~$\ph$
    on\/~$\Ell^\psi_\La(\tr\otimes\tau)$ and the norm of the
    relative Schur multiplier\/~${\Hph}$
    on\/~$\Sch^\psi_{\mkern-3mu\HLa}(\Sch^\psi)$ are equal.
  \item Let\/ $p\ge1$. The complete norm of the relative Fourier
    multiplier\/~$\ph$ on\/~$\Ell^p_\La$ and the complete norm of the
    relative Schur multiplier\/~$\Hph$ on\/~$\Sch^p_{\mkern-3mu\HLa}$ are
    equal:
    \begin{equation*}
      \|\ph\|_{\Fourier_\cb(\Ell^p_\La)}=\|\Hph\|_{\Schur_\cb(\Sch^p_{\mkern-3mu\HLa})}.
    \end{equation*}
  \item The norm of the relative Fourier multiplier\/~$\ph$ on\/~$\Cont_\La$, its complete norm, the norm of the relative Schur
    multiplier\/~$\Hph$ on\/~$\Sch^\infty_{\mkern-3mu\HLa}$, and its complete norm are equal:
    \begin{equation*}
      \|\ph\|_{\Fourier(\Cont_\La)}
      =\|\ph\|_{\Fourier_\cb(\Cont_\La)}
      =\|\Hph\|_{\Schur_\cb(\Sch^\infty_{\mkern-3mu\HLa})}
      =\|\Hph\|_{\Schur(\Sch^\infty_{\mkern-3mu\HLa})}.
    \end{equation*}
\item Suppose that\/ $\La=\Ga$. The norm of the Fourier algebra
  multiplier\/~$\ph$, its complete norm, the norm of the Schur
  multiplier\/~$\Hph$ on\/~$\Sch^\infty$, and its complete norm are
  equal:
    \begin{equation*}
      \|\ph\|_{\Fourier(\mathrm{A}(\Ga))}
      =\|\ph\|_{\Fourier_\cb(\mathrm A(\Ga))}
      =\|\Hph\|_{\Schur_\cb(\Sch^\infty)}
      =\|\Hph\|_{\Schur(\Sch^\infty)}.
    \end{equation*}
  \end{enumerate}
\end{thm}
\begin{proof}
  $(a)$. Combine the argument in Theorem~\ref{trprp} with the matrix
  Szeg\H{o} limit theorem and apply Lemma~\ref{trans1}. 

  $(c)$. Recall that the complete norm of a Schur multiplier~$\Hph$
  on~$\Sch^\infty_{\mkern-3mu\HLa}$ is equal to its norm
  (\cite[Theorem~3.2]{pps89}). Recall also that the norm of a Fourier
  multiplier~$\chi$ on~$\Cont$ is equal to its complete norm, because
  $\Ga$~is amenable. Moreover, it coincides with the norm of $\chi$ in
  $\mathrm{A}(\Ga)$ (\cite[Corollary~1.8]{dh85}). Let $\ph$ be a
  relative contractive Fourier multiplier on $\Cont_\La$; compose it
  with the trivial character of $\Ga$ to obtain a contractive form on
  $\Cont_\La$.  Then, by the Hahn-Banach extension theorem, $\ph$ is
  the restriction of a contractive element $\chi$ in
  $\mathrm{A}(\Ga)$. Now $\chi$ is a completely contractive Fourier
  multiplier on $\Cont$, and so is $\ph$ on $\Cont_\La$. The
  conclusion follows from $(a)$ and $(b)$.
\end{proof}

\section{\texorpdfstring{Local embeddings of~$\Ell^p$ into~$\Sch^p$}{Local embeddings of Lp into Sp}}\label{sec:loc}

\noindent The proof of Theorem~\ref{trprp} can be interpreted as an
embedding of $\Ell^\psi$ into an ultraproduct of finite-dimensional
spaces $\Sch^\psi_n$ that intertwines Fourier and Toeplitz Schur
multipliers. If we restrict ourselves to power functions $\psi\colon
t\mapsto t^p$ with $p\ge 1$, such embeddings are well known and the
proof of Theorem~\ref{trans2} does not need the full strength of the
matrix Szeg\H{o} limit theorem but only the existence of such
embeddings. In this section, we explain two ways to obtain them by
interpolation.

The first way is to extend the classical result that the reduced
$\Cont^*$-algebra~$\Cont$ of a discrete group $\Ga$ has the completely positive
approximation property if $\Ga$ is amenable. We follow the approach
of \cite[Theorem~2.6.8]{bo08}.  Let $\Ga$ be a discrete amenable
group and let $\Ga_\iota$ be a F{\o}lner averaging net of sets. As
above, we denote by $p_\iota$ the orthogonal projection from
$\ell^2_\Ga$ to $\ell^2_{\Ga_\iota}$. Define the
compression~$\phi_\iota$ and the embedding $\psi_\iota$ by
\begin{equation}
  \label{phipsi}
  \begin{aligned}[t]
    \phi_\iota\colon\Cont&\to\borne(\ell^2_{\Ga_\iota})\\
    x&\mapsto p_\iota x p_\iota^*
  \end{aligned}
  \quad\text{and}\quad
  \begin{aligned}[t]
    \psi_\iota\colon\borne(\ell^2_{\Ga_\iota})&\to\Cont\\
    \e_{r,c}&\mapsto(1/\card \Ga_\iota)\la_r\la_{c^{-1}}.
  \end{aligned}
\end{equation}
If we endow $\borne(\ell^2_{\Ga_\iota})$ with the normalised trace,
these maps are unital completely positive, trace preserving (and
normal), and the net $(\psi_\iota\phi_\iota)$ converges pointwise to the
identity of $\Cont$. One can therefore extend them by interpolation to
completely positive contractions on the respective noncommutative
Lebesgue spaces. Recall that $\Ell^p(\borne(\ell^2_{\Ga_\iota}),
(1/\card\Ga_\iota)\tr)$ is
${(\card\Ga_\iota)}^{-1/p}\Sch^p_{\card \Ga_\iota}$. We get a
net of complete contractions
\begin{equation*}
  \tilde \phi_\iota \colon  \Ell^p\to {(\card\Ga_\iota)}^{-1/p}\Sch^p_{\card \Ga_\iota}
  \quad\text{and}\quad
  \tilde \psi_\iota\colon {(\card\Ga_\iota)}^{-1/p}\Sch^p_{\card \Ga_\iota} \to \Ell^p
\end{equation*}
such that $(\tilde\psi_\iota\tilde\phi_\iota)$ converges pointwise
to the identity of $\Ell^p$. Moreover, the definitions~\eqref{phipsi}
show that these maps also intertwine Fourier and Toeplitz Schur
multipliers.

This approach is more canonical, as it allows us to extend the transfer
to vector-valued spaces in the sense of \cite[Chapter
3]{pi98}. Recall that for any hyperfinite semifinite von Neumann
algebra $M$ and any operator space $E$, one can define $\Ell^p(M,E)$.
For $p=\infty$, this space is defined as $M\tens_{\min}E$; for
$p=1$, this space is defined as $M_*^{\mathrm{op}}\hat\tens E$; these spaces form an
interpolation scale for the complex method when $1\le p\le
\infty$. For us, $M$ will be $\borne(\ell^2)$ or the group von
Neumann algebra $\vn$. As the maps $\psi_\iota$ and $\phi_\iota$ are unital
completely positive and trace preserving and normal, they define
simultaneously complete contractions on $M$ and $M_*$. By
interpolation, the maps $\psi_\iota\tens \Id_E$ and $\phi_\iota\tens
\Id_E$ are still complete contractions on the spaces $\Ell_p(E)$ and
$\Sch^p[E]$. Let $\ph\in\C^\Ga$; the transfer shows that the
norm of $\Id_E\otimes\Fourier_\ph$ on~$\Ell^p(E)$ is bounded by the
norm of $\Id_E\otimes\Schur_\Hph$ on~$\Sch^p[E]$ and that their
complete norms coincide. In formulas,
\begin{gather*}
  \|\Id_E\tens\Fourier_\ph\|_{\borne(\Ell^p(E))}
  \le\|\Id_E\tens\Schur_\Hph\|_{\borne(\Sch^p[E])},\\
  \|\Id_E\tens\Fourier_\ph\|_{\cb(\Ell^p(E))}
  =\|\Id_E\tens\Schur_\Hph\|_{\cb(\Sch^p[E])}.
\end{gather*}

The compression $\phi_\iota$ provides a two-sided approximation of an
element $x$, whereas the proof of Theorem~\ref{trprp} uses only a
one-sided approximation.  This subtlety makes a difference in our
second way to obtain embeddings, a direct proof by complex
interpolation.

\begin{prp}
  Let\/ $\Ga$~be a discrete amenable group and let\/ $(\mu_\iota)$~be a
  \emph{Reiter net of means} for\/~$\Ga$:
  \begin{itemize}
  \item each\/~$\mu_\iota$ is a positive sequence summing to\/~1 with
    finite support\/~$\Ga_\iota\subseteq\Ga$ and viewed as a
    diagonal operator from\/~$\ell^2_{\Ga_\iota}$ to\/~$\ell^2_\Ga$,
    so that
    \begin{equation*}
      \|\mu_\iota\|_{\Sch^1}=\sum_{\ga\in\Ga_\iota}{(\mu_\iota)}_{\ga}=1;
    \end{equation*}
  \item the net\/~$(\mu_\iota)$ satisfies, for each\/~$\ga\in\Ga$,
    \emph{Reiter's Property~$P_1$}:
    \begin{equation*}
      \sum_{\beta\in\Ga}{\bigl|{(\mu_\iota)}_{\ga^{-1}\beta}-{(\mu_\iota)}_{\beta}\bigr|}\to0.
    \end{equation*}
  \end{itemize}
  Let\/~$x\in\mat_n\otimes\vn=\vn(\tr\otimes\tau)$ 
  and\/~$p\ge1$. Then
  \begin{equation*}
    \limsup\|x\mu_\iota^{1/p}\|_{\Sch^p_{\vphantom{n}}(\Sch^p_n)}=\|x\|_{\Ell^p(\tr\otimes\tau)}.
  \end{equation*}
\end{prp}

\begin{proof}
  Consider~$x=\sum_{\ga\in\Ga}x_\ga\tens\la_\ga$
  with only a finite number of the~$x_\ga\in\mat_n$ nonzero. As
  \begin{equation*}
    \sum_{\beta\in\Ga}{\bigl|{(\mu_\iota)}_{\ga^{-1}\beta}^{1/2}-{(\mu_\iota)}_{\beta}^{1/2}\bigr|}^2\le
    \sum_{\beta\in\Ga}{\bigl|{(\mu_\iota)}_{\ga^{-1}\beta}-{(\mu_\iota)}_{\beta}\bigr|},
  \end{equation*}
  Property~$P_1$ implies \emph{Property~$P_2$}:
  \begin{equation*}
    \|\la_\ga\mu_\iota^{1/2}-\mu_\iota^{1/2}\la_\ga\|_{\Sch^2}\to0, 
  \end{equation*}
  so that
  \begin{equation*}
    \|x\mu_\iota^{1/2}-\mu_\iota^{1/2}x\|_{\Sch^2(\Sch^2_n)}\to0.
  \end{equation*}
  As the $\mat_n$-valued matrix of~$x$ for the canonical basis of~$\ell^2_\Ga$ is~$(x_{rc^{-1}})_{(r,c)\in\Ga\times\Ga}$,
  \begin{align*}
    \|x\mu_\iota^{1/2}\|_{\Sch^2(\Sch^2_n)}^2
    &=\sum_{(r,c)\in\Ga\times\Ga}\|x_{rc^{-1}}\|_{\Sch^n_2}^2{(\mu_\iota)}_c\\
    &=\sum_{c\in\Ga}{(\mu_\iota)}_c\sum_{r\in\Ga}\|x_{rc^{-1}}\|_{\Sch^n_2}^2\\
    &=\sum_{c\in\Ga}{(\mu_\iota)}_c\|x\|_{\Ell^2(\tr\otimes\tau)}^2=\|x\|_{\Ell^2(\tr\otimes\tau)}^2.
  \end{align*}
By density and continuity, the result extends to all~$x\in \Ell^2(\tr\otimes\tau)$.

  Let us prove now that for~$x\in\Ell^\infty(\tr\otimes\tau)$,
  \begin{equation*}
    \limsup\|x\mu_\iota\|_{\Sch^1(\Sch^1_n)}\le\|x\|_{\Ell^1(\tr\otimes\tau)}.
  \end{equation*}
  The polar decomposition~$x=u|x|$ yields a factorisation~$x=ab$
  with~$a=u\abs{x}^{1/2}$ and~$b=\abs{x}^{1/2}$
  in~$\Ell^\infty(\tr\otimes\tau)$ such that
  \begin{gather*}
    \|a\|_{\Ell^2(\tr\otimes\tau)}=\|b\|_{\Ell^2(\tr\otimes\tau)}=\|x\|_{\Ell^1(\tr\otimes\tau)}^{1/2} \\
    \|a\|_{\Ell^\infty(\tr\otimes\tau)}=\|x\|_{\Ell^\infty(\tr\otimes\tau)}^{1/2}.
  \end{gather*}
  Then
  \begin{math}
    x\mu_\iota=a(b\mu_\iota^{1/2}-\mu_\iota^{1/2}b)\mu_\iota^{1/2}+a\mu_\iota^{1/2}b\mu_\iota^{1/2}
  \end{math},
  so that the Cauchy-Schwarz inequality yields
  \begin{align*}
    \|x\mu_\iota\|_{\Sch^1(\Sch^1_n)}
      &\le\|a\|_{\Ell^\infty(\tr\otimes\tau)}
      \|(b\mu_\iota^{1/2}-\mu_\iota^{1/2}b)\mu_\iota^{1/2}\|_{\Sch^1(\Sch^1_n)}
      +\|a\mu_\iota^{1/2}b\mu_\iota^{1/2}\|_{\Sch^1(\Sch^1_n)}\\
      &\le\|a\|_{\Ell^\infty(\tr\otimes\tau)}
      \|b\mu_\iota^{1/2}-\mu_\iota^{1/2}b\|_{\Sch^2(\Sch^2_n)}
      +\|a\|_{\Ell^2(\tr\otimes\tau)}\|b\|_{\Ell^2(\tr\otimes\tau)}
  \end{align*}
  and therefore our claim. Now complex interpolation yields
  \begin{equation*}
    \limsup\|x\mu_\iota^{1/p}\|_{\Sch^p_{\vphantom{n}}(\Sch^p_n)}\le\|x\|_{\Ell^p(\tr\otimes\tau)}
  \end{equation*}
  for~$x\in\Ell^\infty(\tr\otimes\tau)$ and~$p\in[1,\infty]$.  In
  fact, consider the function~$f(z)=u\abs{x}^{pz}\mu_\iota^z$ analytic
  in the strip~$0<\Im z<1$ and continuous on its closure; then $f(\iu
  t)$~is a product of unitaries for~$t\in\R$, so that
  \begin{math}
    \|f(\iu t)\|_{\vn(\tr\otimes\tau)}=1.
  \end{math}
  Also
  \begin{align*}
    \|f(1+\iu t)\|_{\Sch^1(\Sch^1_n)}
    &=\|\abs{x}^p\mu_\iota\|_{\Sch^1(\Sch^1_n)}.
  \end{align*}
  As $\Sch^p(\Sch^p_n)$~is the complex interpolation space~$(\Sch^\infty(\mat_n),\Sch^1(\Sch^1_n))_{1/p}$,
  \begin{align*}
    \|x\mu_\iota^{1/p}\|_{\Sch^p_{\vphantom{n}}(\Sch^p_n)}=
    \|f(1/p)\|_{\Sch^p_{\vphantom{n}}(\Sch^p_n)} \le
    \||x|^p\mu_\iota\|_{\Sch^1(\Sch^1_n)}^{1/p}.
  \end{align*}
Then, taking the upper limit and using the estimate on~$\Sch^1(\Sch^1_n)$,
 \begin{align*}
\limsup \|x\mu_\iota^{1/p}\|_{\Sch^p_{\vphantom{n}}(\Sch^p_n)}&\le 
\limsup \||x|^p\mu_\iota\|_{\Sch^1(\Sch^1_n)}^{1/p}\\ &\le
\||x|^p\|_{\Ell^1(\tr\otimes\tau)}^{1/p}=\|x\|_{\Ell^p(\tr\otimes\tau)}.
\end{align*}
  The reverse inequality is obtained by duality; first note that for~$y\in{\vn(\tr\otimes\tau)}$, 
\begin{equation*}\lim \tr\otimes\tr(y\mu_\iota)= \tr\otimes\tau(y).\end{equation*}
With the above notation and the inequality for~$p'$,
\begin{align*}
 \|x\|_{\Ell^p(\tr\otimes\tau)}^p&=\tau (|x|^p)=\lim \tr |x|^p\mu_\iota
 = \lim\tr \mu_\iota^{1-1/p}|x|^{p-1}u^*x\mu_\iota^{1/p}\\
&\le
\limsup \|\mu_\iota^{1-1/p}|x|^{p-1}\|_{\Sch^{p'}_{\vphantom{n}}(\Sch^{p'}_n)}
\|x\mu_\iota^{1/p}\|_{\Sch^p_{\vphantom{n}}(\Sch^p_n)}\\&=
\limsup \||x|^{p-1}\mu_\iota^{1-1/p}\|_{\Sch^{p'}_{\vphantom{n}}(\Sch^{p'}_n)}
\|x\mu_\iota^{1/p}\|_{\Sch^p_{\vphantom{n}}(\Sch^p_n)} \\&\le
\||x|^{p-1}\|_{\Ell^{p'}(\tr\otimes\tau)}\limsup\|x\mu_\iota^{1/p}\|_{\Sch^p_{\vphantom{n}}(\Sch^p_n)},
\end{align*}
 so that
  \begin{equation*}
    \limsup\|x\mu_\iota^{1/p}\|_{\Sch^p_{\vphantom{n}}(\Sch^p_n)}=\|x\|_{\Ell^p(\tr\otimes\tau)}^p.\qedhere
  \end{equation*}
\end{proof}
\begin{rem}Let $\mu$~be any positive diagonal operator with $\tr
  \mu=1$ and $p\ge 2$; then $\|x\mu^{1/p}\|_{\Sch^p(\Sch^p_n)}\le \|x\|_{\Ell^p}$
  for all $x\in \vn(\tr\otimes\tau)$. The Reiter condition is
  only necessary to go below exponent~2.
\end{rem}

We could also have used interpolation with a two-sided
approximation by Reiter means. We would have obtained
\begin{equation*}
  \limsup\|\mu_\iota^{1/2p}x\mu_\iota^{1/2p}\|_{\Sch^p_{\vphantom{n}}(\Sch^p_n)}
  =\|x\|_{\Ell^p(\tr\otimes\tau)}.
\end{equation*} 
This formula is in the spirit of the first approach of this section.

\section{Transfer of lacunary sets into lacunary matrix patterns}
\label{sec:lac}

As a first application of Theorem~\ref{trans2}, let us mention that it
provides a shortcut for some arguments in \cite{ha98}, as it permits us to
transfer lacunary subsets of a discrete group $\Ga$ into lacunary matrix
patterns in $\Ga\times\Ga$.  Let us first introduce the following
terminology.

\begin{dfn}\label{unclp:def}
  Let $\Ga$~be a discrete group and~$\La\subseteq\Ga$.  Let $X$~be the
  reduced $\Cont^*$-algebra $\Cont$ of~$\Ga$ or its noncommutative
  Lebesgue space~$\Ell^p$ for~$p\in\iod[1,\infty]$.
  \begin{enumerate}
  \item The set~$\La$ is \emph{unconditional} in~$X$ if the Fourier
    series of every~$x\in X_\La$ converges unconditionally; i.e., there is a
    constant~$D$ such that
    \begin{equation*}
      \biggnorm{\sum_{\ga\in\La'}x_\ga\eps_\ga\la_\ga}_{X}
      \le D\|x\|_{X}
    \end{equation*}
    for finite~$\La'\subseteq\La$
    and~$\eps_\ga\in\T$. The minimal constant~$D$ is the
    \emph{unconditional constant of~$\La$ in~$X$}.
  \item If~$X=\Cont$, let~$\tilde X=\Sch^\infty\otimes\Cont$; if~$X=\Ell^p$, let~$\tilde X=\Ell^p(\tr\otimes\tau)$. The set~$\La$ is \emph{completely unconditional} in~$X$ if the
    Fourier series of every~$x\in\tilde X_\La$ converges
    unconditionally; i.e., there is a constant~$D$ such
    that
  \begin{equation*}
    \biggnorm{\sum_{\ga\in\La'}x_\ga\otimes\eps_\ga\la_\ga}_{\tilde X}
    \le D\|x\|_{\tilde X}
  \end{equation*}
  for finite~$\La'\subseteq\La$
  and~$\eps_\ga\in\T$. The minimal constant~$D$ is the
  \emph{complete unconditional constant of~$\La$ in~$X$}.
  \end{enumerate}
\end{dfn}

Unconditional sets in~$\Ell^p$ have been introduced as ``$\Lambda(p)$
sets'' in \cite[Definition~1.1]{ha98} for $p>2$. If $\Ga$~is abelian,
they are Walter Rudin's $\Lambda(p)$ sets if $p>2$ and his
$\Lambda(2)$ sets if $p<2$ (see \cite{ru60,bo01}).  Asma Harcharras
(\cite[Definition~1.5, Comments~1.9]{ha98}) called completely
unconditional sets in~$\Ell^p$ ``$\Lambda(p)_{\cb}$ sets''
if~$p\in\io[2,\infty]$, and ``$\mathrm{K}(p)_{\cb}$ sets''
if~$p\in\iog[1,2]$; her definitions are equivalent to ours by the
noncommutative Khinchin inequality. 

Sets that are unconditional in $\Cont$ have been introduced as
``unconditional Sidon sets'' in \cite{bo81a}. If $\Ga$ is amenable,
Fourier multipliers are automatically c.b.\ on $\Cont_\La$, so that
such sets are automatically completely unconditional in~$\Cont$, and
there are at least three more equivalent definitions for the
counterpart of Sidon sets in an abelian group. If $\Ga$ is
nonamenable, these definitions are no longer all equivalent, and our
notion of completely unconditional sets in~$\Cont$ corresponds to Marek Bo\.zejko's ``c.b.\ Sidon sets.''

\begin{dfn}
  \label{uncsp:def}
  Let $1\le p\le\infty$ and $I$ be a subset of the product $\Row\times\Col$ of two index sets.
    \begin{enumerate}
  \item The set~$I$ is \emph{unconditional} in the Schatten-von-Neumann class~$\Sch^p$ associated with $\borne(\ell^2_\Col,\ell^2_R)$ if the matrix
    representation of every~$x\in\Sch^p_I$ converges
    unconditionally; i.e., there is a constant~$D$ such that
    \begin{equation*}
      \Bigl\|\sum_{q\in I'}x_q\eps_q\e_q\Bigr\|_p \le D
      \|x\|_p
    \end{equation*}
    for finite~$I'\subseteq I$
    and~$\eps_q\in\T$. The minimal constant~$D$ is the
    \emph{unconditional constant of $I$ in $\Sch^p$}.
  \item The set $I$ is \emph{completely} unconditional in $\Sch^p$ if
    the matrix representation of every~$x\in\Sch^p_I(\Sch^p_{\vphantom{I}})$
    converges unconditionally; i.e., there is a constant~$D$ such that
    \begin{equation*}\label{cu}
      \Bigl\|\sum_{q\in I'}x_q\tens\eps_q\e_q\Bigr\|_p \le D
      \|x\|_p
    \end{equation*}
    for finite~$I'\subseteq I$
    and~$\eps_q\in\T$. The minimal constant~$D$ is the
    \emph{complete unconditional constant of $I$ in $\Sch^p$}.
  \end{enumerate}
\end{dfn}

Harcharras called unconditional and
completely unconditional sets in~$\Sch^p$ ``$\sigma(p)$ sets'' and
``$\sigma(p)_{\cb}$ sets'', respectively (\cite[Definitions
  4.1~and~4.4, Remarks~4.6$\,(iv)$]{ha98}); she supposed $p<\infty$, so that  her definitions are
equivalent to ours by the noncommutative Khin\-chin inequality.

\begin{prp}\label{trans:set}
  Let\/ $\Ga$~be a discrete group. Let\/ $\La\subseteq\Ga$ and consider
  the associated Toeplitz set\/
  $\HLa=\{(r,c)\in\Ga\times\Ga:\row\col^{-1}\in \La\}$. Let\/ $p\in\iod[1,\infty]$.
  \begin{enumerate}
  \item If\/ $\Ga$ is amenable, then\/ $\La$ is unconditional in\/ $\Ell^p$
    if\/ $\HLa$ is unconditional in\/ $\Sch^p$.
  \item If\/ $\La$ is completely unconditional in\/ $\Ell^p$, then\/ $\HLa$
    is completely unconditional in\/ $\Sch^p$. The converse holds if\/
    $\Ga$ is amenable.
  \end{enumerate}
\end{prp}

\begin{proof}
  The first part of $(b)$ follows by the argument of the proof of
  \cite[Proposition~4.7]{ha98}; let us sketch it. Consider the
  isometric embedding of the space $\Sch^p_{\mkern-3mu\HLa}(\Sch^p)$ in
  $\Ell^p_\La(\tr\otimes\tr\otimes\tau)$ that is given in the proof of
  Lemma~\ref{trans1} and apply the equivalent Definition~1.5 in
  \cite{ha98} of the complete unconditionality of~$\La$: this gives
  the complete unconditionality of~$\HLa$ in the equivalent
  Definition~4.4 in \cite{ha98}.

  Unconditionality in $\Ell^p$ expresses the uniform boundedness of
  relative unimodular Fourier multipliers on $\Ell^p_\La$; complete
  unconditionality expresses their uniform complete boundedness.
  Unconditionality in $\Sch^p$ expresses the uniform boundedness of
  relative unimodular Schur multipliers on $\Sch^p_{\mkern-3mu\HLa}$; complete
  unconditionality expresses their uniform complete boundedness. The
  second part of $(b)$ follows therefore from
  Theorem~\ref{trans2}$\,(b)$ and $(a)$ follows from
  Theorem~\ref{trprp}.
\end{proof}

\begin{rem}
  This transfer does not pass to the limit $p=\infty$ in~$(b)$ and is
  void in~$(a)$. Nicholas Varopoulos proved that unconditional
  sets in $\Sch^\infty$ are finite unions of patterns whose rows or
  whose columns contain at most one element, and this excludes sets of
  the form $\HLa$ for any infinite $\La$ (\cite[Theorem~4.2]{va69}, see
  \cite[\S\,5]{ne06} for a reader's guide).
\end{rem}

\begin{rem}
  See \cite[Remark~11.3]{ne06} for an illustration of
  Proposition~4.3$\,(b)$ in a particular context.
\end{rem}

\begin{rem}
  Let $p$ be an even integer greater than or equal to 4. The existence of a
  $\sigma(p)_\cb$ set that is not a $\sigma(q)$ set for any $q>p$
  (\cite[Theorem~4.9]{ha98}) becomes a direct consequence of Walter
  Rudin's construction (\cite[Theorem~4.8]{ru60}) of a $\Lambda(p)$ set
  that is not a $\Lambda(q)$ set for any $q>p$, because this set has
  property $\mathrm B(p/2)$ (\cite[Definition~2.4]{ha98}) and is
  therefore $\Lambda(p)_\cb$ by \cite[Theorem~1.13]{ha98} (in fact, it
  is even ``1-unconditional'' in $\Ell^p$ because $\mathrm B(p/2)$ is
  ``$p/2$-independence'' (\cite[\S\,11]{ne06})).
\end{rem}

\begin{rem}
  In the same way, \cite[Theorem~5.2]{ha98} becomes a mere
  reformulation of \cite[Proposition 3.6]{ha98} if one remembers that
  the Toeplitz Schur multipliers are 1-complemented in the Schur
  multipliers for an amenable discrete group and for all classical
  norms.  Basically, results on $\Lambda(p)_{\cb}$ sets produce results
  on $\sigma(p)_{\cb}$ sets.
\end{rem}

Let us now estimate the complete unconditional constant of sumsets. In
the case~$\Ga=\Z$, Harcharras (\cite[Prop.~2.8]{ha98}) proved that
a completely unconditional set in~$\Ell^p$ cannot contain the sumset
of characters~$A+A$ for arbitrary large finite sets~$A$. In
particular, if $\La\supseteq A+A$ with $A$ infinite, then $\La$~is not
a completely unconditional set in $\Ell^p$. Thus, her proof provided examples
of~$\Lambda(p)$ sets that are not $\Lambda(p)_\cb$ sets.

We generalise Harcharras' result in two directions. Compare~\cite[\S\,1.4]{lr75}.

\begin{prp}\label{har:lp}
  Let\/ $\Ga$~be a discrete group and\/~$p\ne2$. A completely
  unconditional set in\/~$\Ell^p$ cannot contain the sumset of two
  arbitrarily large sets. More precisely, let\/ $\Row$~and\/~$\Col$ be
  subsets of\/~$\Ga$ with\/~$\card{\Row}\ge n$ and\/ $\card{\Col}\ge n^3$. Then, for
  any\/~$p\ge1$, the complete unconditional constant of the
  sumset\/~$\Row\Col$ in\/~$\Ell^p$ is at least\/~$n^{|1/2-1/p|}$.
\end{prp}

\begin{proof} Let $r_1,\dots,r_{n}$ be pairwise
  distinct elements in $\Row$. We shall select inductively elements
  $c_1,\dots,c_{n}$ in $\Col$ such that the
  $r_ic_j$ are pairwise distinct.
Assume there are $c_1,\dots,c_{m-1}$ such that the induction hypothesis
\begin{equation*}
  \forall\, i,k\le n\ \forall\,j,l\le m-1\quad
  (i,j)\ne(k,l)\ \imp\ \row_i\col_j\ne\row_k\col_l.
\end{equation*}
holds. We are looking for an element $c_m\in\Col$ such that 
\begin{equation*}
    \forall\, i,k\le n\ 
    \forall\, l\le m-1\quad 
    \row_i\col_m\ne\row_k\col_l.
  \end{equation*}
  Such an element exists as long as $m\le n$, because the set $\{r_i^{-1}r_kc_l:i,k\le n,\,\allowbreak l\le m-1\}$ has at most $\bigl(n(n-1)+1\bigr)(m-1)<n^3$ elements.  

  The end of the proof is the same as Harcharras'.  The unconditional
  constant of the canonical basis of elementary matrices in~$\Sch^p_n$
  is~$n^{|1/2-1/p|}$; in particular, there is an unimodular Schur
  multiplier~$\bph$ on~$\Sch^p_n$ of norm~$n^{|1/2-1/p|}$ (which is
  also its complete norm, by the way; see \cite[Lemma~8.1.5]{pi98}).
  Let $\La$ be the sumset $\{r_ic_j:i,j\le n\}$;
  as the $r_ic_j$ are pairwise distinct, we may define a
  sequence~$\ph\in\C^\La$ by~$\ph_{\row_i\col_j}=\bph_{i,j}$. By
  Remark~\ref{transg}, the complete norm of the Fourier
  multiplier~$\ph$ on~$\Ell^p_\La$ is bounded below by the complete
  norm of the Schur multiplier~$\bph$ on~$\Sch^p_I$.
\end{proof}
\begin{exa}
  $\La=\{2^i-2^j:i>j\}$~does not form a complete~$\Lambda(p)$ set for
  any~$p\ne2$.  Indeed, \ $\{2^i-2^j\}=\La\cup-\La$ does not, and if
  $\La$ did, then so would $-\La$ and~$\La\cup-\La$.
\end{exa}
  
\section{\texorpdfstring{Toeplitz Schur multipliers on~$\Sch^p$ for~$p<1$}{Toeplitz Schur multipliers on Sp for p<1}}\label{sec:tsp1}

When~$0<p<1$, a complete characterisation of bounded Schur multipliers of 
Toeplitz type has been obtained by Alexey Alexandrov and Vladimir Peller in \cite[Theorem~5.1]{AlPe02}.
This result was an easy consequence of their deep results on Hankel Schur 
multipliers. The transfer approach provides a direct proof.

\begin{cor}\label{AP}
  Let\/~$0<p<1$. Let\/ $\Ga$~be a discrete abelian group with dual
  group\/~$G$. Let\/ $\ph$~be a sequence indexed by\/~$\Ga$ and define the
  associated Toeplitz matrix\/~$\Hph\in\C^\HLa$
  by\/~$\Hph(r,c)=\ph(\row\col^{-1})$ for\/~$(\row,\col)\in\Ga\times\Ga$.
  Then the following are equivalent:
\begin{enumerate}
\item the sequence\/~$\ph$ is the Fourier transform of an atomic measure\/~$\mu=\sum a_g \delta_{g}$ on\/~$G$ with\/~$\sum\abs{a_g}^p\le 1$;
\item the Fourier multiplier\/~$\ph$ is contractive on\/~$\Ell^p$;
\item the Fourier multiplier\/~$\ph$ is contractive on\/~$\Ell^p(\Sch^p)$;
\item the Schur multiplier\/~$\Hph$ is contractive on\/~$\Sch^p$;
\item the Schur multiplier\/~$\Hph$ is contractive on\/~$\Sch^p(\Sch^p)$.
\end{enumerate}
 \end{cor}

\begin{proof}
  The implication~$(d)\Rightarrow(b)$ follows from
  Theorem~\ref{trprp}. The equivalence~$(c)\Leftrightarrow(e)$
  follows from Theorem~\ref{trans2}$\,(a)$. The
  characterisation~$(a)\Leftrightarrow(b)$ is an old result of Daniel Oberlin
  (\cite{Ober76}). It is plain that $(e)\Rightarrow(d)$.  At last, \
  $(a)\Rightarrow(c)$~is obvious by the~$p\mkern1mu$-triangular
  inequality.
\end{proof}

\begin{rem}\label{b-cb}
  As a consequence, we get that the norm of a Toeplitz Schur
  multiplier on~$\Sch^p(\Sch^p)$ coincides with its norm on~$\Sch^p$
  when~$p<1$. If $p\in\{1,2,\infty\}$, this holds for every Schur
  multiplier. Let $p\in\io[1,2]\cup\io[2,\infty]$. Then we still do
  not know whether Schur multipliers are automatically c.b.\
  on~$\Sch^p$. But from \cite[Proposition~8.1.3]{pi98}, we know that
  $(b)$~and~$(c)$ are not equivalent: if~$\Ga$ is an infinite abelian
  group, there is a bounded Fourier multiplier on~$\Ell^p$ that is not
  c.b.  This counterexample is easy to describe: if an infinite set
  $A\subseteq\Ga$ is lacunary enough, the sumset $A+A$ is
  unconditional in~$\Ell^p$ (see \cite[Theorem~5.13]{lr75}). By
  Proposition~\ref{har:lp}, it cannot be completely unconditional.  In
  particular, this shows that in Remark~\ref{trrem} we cannot
  remove the restriction to truncated Toeplitz matrices in the
  computation of the Schur multiplier norm; that is,
  $(b)\Rightarrow(d)$ does not hold.
\end{rem}

\begin{rem}
  Our questions may also be addressed in the case of a compact group like~$\T$.
  A measurable function~$\ph$ on~$\T$ defines
  \begin{itemize}
  \item the Fourier multiplier on measurable functions on~$\T$ by
    $x\mapsto\ph x$;
  \item the Schur multiplier on integral operators on $\Ell^2(\T)$ with kernel a
    measurable function~$x$ on $\T\times\T$ by $x\mapsto\Hph x$, where
    $\Hph(z,w)=\ph(zw^{-1})$.
  \end{itemize}
  Victor Olevskii (\cite{ol96}) constructed a continuous function~$\ph$
  that defines a boun\-ded Fourier multiplier on the space of functions
  with $p$-summable Fourier series endowed with the norm given by
  $\norm{x}=\bigl(\sum\abs{\hat x(n)}^p\bigr)^{1/p}$ for every
  $p\in\io[1,\infty]$, while the corresponding Schur
  multiplier is not bounded on the Schatten-von-Neumann class $\Sch^p$
  of operators on $\Ell^2(\T)$ for any $p\in\io[1,2]\cup\io[2,\infty]$.

\end{rem}

\section{{The Riesz projection and the Hilbert
    transform}}\label{sec:rpht}

In this section, we concentrate on~$\Ga=\Z$, the dual group of~$\T$.

\begin{prp}
  \label{rhcb}
  Let\/ $\rh$~be a linear combination of the identity and the upper
  triangular projection of\/~$\N\times\N$; i.e., there are\/~$z,w\in\C$ so that\/
  $\rh_{i,j}=z$ if\/~$i\le j$ and\/~$\rh_{i,j}=w$ if\/~$i>j$.  Then the norm
  of the Schur multiplier\/~$\rh$ on\/~$\Sch^\psi$ coincides with the norm
  of the Schur multiplier\/~$\rh$ on\/ $\Sch^\psi(\Sch^\psi)$.
\end{prp}

\begin{proof}
  Let $a\in\Sch^\psi_m(\Sch^\psi_n)$; \ $a$~may be considered as an~$m\times
  m$ matrix~$(a_{ij})$ whose entries~$a_{ij}$ are $n\times n$
  matrices, and may be identified with the block matrix
  \begin{equation*}
    \tilde{a}=
    \begin{pmatrix}
      0&a_{11}&0&a_{12}&\cdots\\
      0&0&0&0&\cdots\\
      0&a_{21}&0&a_{22}&\cdots\\
      0&0&0&0&\cdots\\
      \vdots&\vdots&\vdots&\vdots&\ddots
    \end{pmatrix}.
  \end{equation*}
In this identification, \ $\Id_{\Sch^\psi_n}\otimes\Schur_\rh(a)$ is~$\Schur_\rh(\tilde{a})$.
\end{proof}

The Hilbert transform~$\Ht$ is the Schur multiplier obtained by
choosing $z=-1$ and $w=1$. The upper triangular operators in~$\Sch^p$ can be
seen as a noncommutative $\Hardy^p$ space, and $\Ht$ corresponds
exactly to the Hilbert transform in this setting (see
\cite{Ran98,MaWe98}).  Using classical results on $\Hardy^p$ spaces,
all Hilbert transforms are c.b.\ for~$1<p<\infty$
(see \cite{Zs80,Ran98,MaWe98}).

On the circle~$\T$, the classical Hilbert transform~$H$ corresponds to
the Fourier multiplier given by the sign function (with the
convention~$\sgn(0)=1$), and its norm on~$\Ell^p$ is
\begin{math}
  \cot \left(\pi/2\max(p,p')\right)=\csc(\pi/p)+\cot(\pi/p)
\end{math}
for~$1<p<\infty$. The story of the computation of this norm starts
with a paper by Israel Gohberg and Naum Krupnik (\cite{Gk68}) for $p$
a power of~2.  The remaining cases were handled by Stylianos
Pichorides (\cite{Pi72}) and Brian Cole (see \cite{Gam78})
independently.  The best results in this subject are those of Brian
Hollenbeck, Nigel Kalton, and Igor Verbitsky (\cite{HoKaVe03}), but
they rely on complex variable methods that are not available in the
operator-valued case.  When $p$~is a power of~2 (or its conjugate), a
combination of arguments of Gohberg and Krupnik (\cite{GoKr92}) with
some of L\'aszl\'o Zsid\'o (\cite{Zs80}) yields the following result.
\begin{thm}\label{normh}
  Let\/ $p\in\io[1,\infty]$. The norm and the complete norm of the
  Hilbert transform\/~$\Ht$ on\/~$\Sch^p$ coincide with the complete norm
  of the Hilbert transform\/~$H$ on\/~$\Ell^p$: if\/ $\Hsgn(i,j)=\sgn(i-j)$
  for\/~$i,j\ge1$,
  \begin{equation*}
    \|\Hsgn\|_{\Schur(\Sch^p)}=\|\Hsgn\|_{\Schur_\cb(\Sch^p)}=\|\sgn\|_{\Fourier_\cb(\Ell^p)}.
  \end{equation*}
  If\/ $p$~is a power of~2, then these norms coincide with the norm
  of\/~$H$ on\/~$\Ell^p$:
  \begin{equation*}
    \|\Hsgn\|_{\Schur(\Sch^p)}=\|\Hsgn\|_{\Schur_\cb(\Sch^p)}=\|\sgn\|_{\Fourier_\cb(\Ell^p)}=\|\sgn\|_{\Fourier(\Ell^p)}=\cot({\pi}/{2p}).
  \end{equation*}
\end{thm}
\begin{proof}
  Let $p\ge 2$. The norm of~$H$ on~$\Ell^p$ is~$\cot({\pi}/{2p})$ and
  the three other norms are equal by the transfer theorem~\ref{trans2}
  and the above proposition. We only need to compute the complete norm
  of~$H$. Let $\tilde H=\Id_{\Sch^p}\otimes H$~be the Hilbert
  transform on~$\Ell^p(\tr\otimes\tau)$. We shall use Mischa Cotlar's
  trick to go from~$\Ell^p$ to~$\Ell^{2p}$: the equality
  \begin{math}
    (\sgn i\sgn j)+1=\sgn(i+j)\*(\sgn i+\sgn j)
  \end{math}
  shows that
\begin{equation}\label{cotlar}
  (\tilde Hf)(\tilde Hg)+fg=\tilde H\bigl((\tilde Hf)g+f(\tilde Hg)\bigr).
\end{equation}
\noindent \emph{Step 1.} The function~$\sgn$ is not odd, because of
its value in~0; this can be fixed in the following way. Let~$\La=2\Z
+1$. The norm of~$\tilde H$ on~$\Ell^p(\tr\otimes\tau)$ is equal to
its norm on~$\Ell^p_\La(\tr\otimes\tau)$. In fact, let
$D$~be defined by~$Df(z)=zf(z^2)$; \ $D$~is a complete isometry
on~$\Ell^p$ with range~$\Ell^p_\La$ that commutes with~$H$.

\noindent \emph{Step 2.} Let $S$~be the real subspace
of~$\Ell_\La^p(\tr\otimes\tau)$ consisting of functions with values
in~$\Sch^p$ so that $f(z)$~is selfadjoint for almost all~$z\in \T$.
Let us apply Vern Paulsen's off-diagonal trick (\cite[Lemma~8.1]{pa02})
to show that the norm of~$\tilde H$ on~$\Ell^p$ is equal to its norm
on~$S$.  Let~$f\in\Ell^p_\La(\tr\otimes\tau)$. Identifying
$\Sch^p_2(\Sch^p_{\vphantom{2}})$ with~$\Sch^p$,
\begin{equation*}
  g(z)= 
  \begin{pmatrix}
    0&f(z)\\
    f(z)^* &0
  \end{pmatrix}
\end{equation*}
defines an element of~$S$.  As the adjoint operation is isometric on~$\Sch^p$,
\begin{equation*}
\|g\|_{S}= 2^{1/p}\|f\|_{\Ell^p(\tr\otimes\tau)}.
\end{equation*}
Let us now consider
\begin{equation*}
  \tilde Hg
  =
  \begin{pmatrix}
    0&\tilde Hf\\
    \tilde H(f^*)&0
  \end{pmatrix}.
\end{equation*}
As $0\notin\La$ by Step~1, the equality~$\sgn(-i)=-\sgn i$ holds
for~$i\in\La$: this yields that $\tilde H(f^*)=-(\tilde
Hf)^*$. Therefore
\begin{equation*}
  \|\tilde Hg\|_S=2^{1/p}\|\tilde Hf\|_{\Ell^p(\tr\otimes\tau)}.
\end{equation*}
 
\noindent \emph{Step 3.} Let $u_p$~be the norm of~$\tilde H$
on~$\Ell^p(\tr\otimes\tau)$; then $u_{2p}\le u_p+\sqrt{1+u_p}$. It
suffices to prove this estimate for~$f\in S$, and by approximation we
may suppose that~$f$ is a finite linear combination of terms~$a_i\tens
z^i+a_i^*\tens z^{-i}$ with~$a_i$ finite matrices. Note that $\tilde
Hf=-(\tilde Hf)^*$. Formula~\eqref{cotlar} with~$f=g$ combined with
H\"older's inequality yields
\begin{equation*}
  \|(\tilde Hf)^2\|_{\Ell^p(\tr\otimes\tau)}
  \le\|f^2\|_{\Ell^p(\tr\otimes\tau)}
  +2u_p\|f\|_{\Ell^{2p}(\tr\otimes\tau)}\|\tilde Hf\|_{\Ell^{2p}(\tr\otimes\tau)}.
\end{equation*}

Since $f$~and~$\tilde Hf$ take normal values,
\begin{gather*}
  \|f^2\|_{\Ell^p(\tr\otimes\tau)}=\|f\|_{\Ell^{2p}(\tr\otimes\tau)}^2\\
  \|(\tilde Hf)^2\|_{\Ell^p(\tr\otimes\tau)}=\|\tilde
  Hf\|_{\Ell^{2p}(\tr\otimes\tau)}^2.
\end{gather*}
Therefore, if
$\|f\|_{\Ell^{2p}(\tr\otimes\tau)}=1$, \ $\|\tilde
Hf\|_{\Ell^{2p}(\tr\otimes\tau)}$ must be smaller than the bigger root of~$t^2-2u_pt-1$; that is,
\begin{equation*}
  \|\tilde Hf\|_{\Ell^{2p}(\tr\otimes\tau)}^2\le u_p+\sqrt{u_p^2+1}\text{ and }u_{2p}\le u_p+\sqrt{u_p^2+1}.
\end{equation*}

\noindent \emph{Step 4.} The multiplier~$H$ is an isometry on~$\Ell^2(\tr\otimes\tau)$, so that $u_2=1=\cot(\pi/4)$. As
$\cot(\th/2)=\cot\th+\sqrt{\cot^2\th+1}$ for~$\th\in\io[0,\pi]$, we
conclude by induction.
\end{proof}

Unfortunately, we cannot deal with other values of~$p>2$ by this method.

The Riesz projection~$\Rt$ is the Schur multiplier obtained by
choosing $z=0$ and $w=1$ in Proposition~\ref{rhcb}. It is the
projection on the upper triangular part. On the circle, the classical
Riesz projection~$T$, that is the projection onto the analytic part,
corresponds to the Fourier multiplier given by the indicator
function~$\un_{\Z^+}$ of nonnegative integers; its norm
on~$\Ell^p$, as computed by Hollenbeck and Verbitsky
(\cite{HoVe00}), is
\begin{math} \csc(\pi/p)\end{math}. As for the Hilbert transform, we
know that the norm and the complete norm of~$\Rt$ on~$\Sch^p$ are
equal and coincide with the complete norm of~$T$ on~$\Ell^p$, but,
to the best of our knowledge, there is no simple formula
like~\eqref{cotlar} to go from exponent~$p$ to~$2p$. We only obtained
the following computation.

\begin{prp}\label{normt}
  Let\/ $p\in\io[1,\infty]$. The norm and the complete norm of the
  Riesz projection\/~$\Rt$ on\/~$\Sch^p$ coincide with the complete
  norm of the Riesz projection\/~$T$ on\/~$\Ell^p$: if\/
  $\Hun(i,j)=\un_{\Z^+}(i-j)$ for\/~$i,j\ge1$,
  \begin{equation*}
    \|\Hun\|_{\Schur(\Sch^p)}=\|\Hun\|_{\Schur_\cb(\Sch^p)}=\|\un_{\Z^+}\|_{\Fourier_\cb(\Ell^p)}.
  \end{equation*}
  If\/~$p=4$, then these norms coincide with the norm of\/~$T$ on\/~$\Ell^p$:
  \begin{equation*}
    \|\Hun\|_{\Schur(\Sch^4)}=\|\Hun\|_{\Schur_\cb(\Sch^4)}=\|\un_{\Z^+}\|_{\Fourier_\cb(\Ell^4)}=\|\un_{\Z^+}\|_{\Fourier(\Ell^4)}=\sqrt2.
  \end{equation*}
\end{prp}

\begin{proof}
  We shall compute the norm of~$\Rt$ on~$\Sch^4$. Let $x$~be a finite
  upper triangular matrix and let $y$~be a finite strictly lower
  triangular matrix. We have to prove that
\begin{equation*}
  \sqrt 2 \|x+y\|_{\Sch^4}\ge  \|x\|_{\Sch^4}.
\end{equation*}
Let us make the obvious estimates on~$\Sch^2$ and use the fact that
the adjoint operation is isometric:
\begin{equation*}
  \|\Rt (xx^*)\|_{\Sch^2}
  =\|\Rt ((x+y)x^*)\|_{\Sch^2} 
  \le\|x+y\|_{\Sch^4}\|x\|_{\Sch^4},
\end{equation*}
and similarly, 
\begin{equation*}
  \| (\Id-\Rt) (xx^*)\|_{\Sch^2}=\| (\Id -\Rt)
  (x(x+y)^*)\|_{\Sch^2} \le \|x\|_{\Sch^4} \|x+y\|_{\Sch^4}.
\end{equation*}
As $\Rt$~and~$\Id-\Rt$ have orthogonal ranges,
\begin{equation*}
\|x\|_{\Sch^4}^4=\|xx^*\|_{\Sch^2}^2=\| (\Id-\Rt) (xx^*)\|_{\Sch^2}^2+
\| \Rt (xx^*)\|_{\Sch^2}^2\le 2 \|x\|_{\Sch^4}^2 \|x+y\|_{\Sch^4}^2.\qedhere
\end{equation*}
\end{proof}

\section{Unconditional approximating sequences}\label{sec:uap}

The following definition makes sense for general operator spaces, but
we choose to state it only in our specific context.

\begin{dfn}\label{uap:def}
  Let $\Ga$~be a discrete group and~$\La\subseteq\Ga$. Let $X$~be the
  reduced $\Cont^*$-algebra of~$\Ga$ or its noncommutative Lebesgue
  space~$\Ell^p$ for~$p\in\iod[1,\infty]$.
  \begin{enumerate}
  \item A sequence~$(T_k)$ of operators on~$X_\La$ is an
    \emph{approximating sequence} if each~$T_k$ has finite rank and~$T_kx\to x$ for every~$x\in X_\La$. It is a \emph{complete}
    approximating sequence if the~$T_k$ are uniformly c.b. If $X_\La$
    admits a complete approximating sequence, then $X_\La$ enjoys
    the \emph{c.b.\ approximation property}.
  \item The \emph{difference sequence}~$(\Delta T_k)$ of a sequence~$(T_k)$ is given by~$\Delta T_1=T_1$ and $\Delta
    T_k=T_k-T_{k-1}$ for~$k\ge2$. An approximating sequence~$(T_k)$ is \emph{unconditional} if the operators
    \begin{equation}
      \label{cuap:eq}
      \sum_{k=1}^n\eps_k\Delta T_k\quad\text{with~$n\ge1$ and~$\eps_k\in\{-1,1\}$}
    \end{equation}
    are uniformly bounded on~$X_\La$; then~$X_\La$ enjoys the
    \emph{unconditional approximation property}.
  \item An approximating sequence~$(T_k)$ is \emph{completely}
    unconditional if the operators in~\eqref{cuap:eq} are uniformly
    c.b.\ on~$X_\La$; then $X_\La$ enjoys the
    \emph{complete} unconditional approximation property. The minimal
    uniform bound of these operators is the \emph{complete
      unconditional constant} of~$X_\La$.
  \end{enumerate}
\end{dfn}

We may always suppose that a complete approximating sequence
on~$\Cont_\La$ is a Fourier multiplier sequence (see
\cite[Theorem~2.1]{hk94}). We may also do so on~$\Ell^p_\La$ if
$\Ell^\infty$ has the so-called QWEP (see \cite[Theorem~4.4]{jr03}).
More precisely, the following proposition holds.

\begin{prp}
  \label{fm}
  Let\/ $\Ga$ be a discrete group and\/ $\La\subseteq\Ga$. Let\/ $X$~either
  be its reduced\/ $\Cont^*$-algebra or its noncommutative Lebesgue
  space\/~$\Ell^p$, where\/ $p\in\iod[1,\infty]$ and\/ $\Ell^\infty$ has the
  QWEP\@. If\/ $X_\La$ enjoys the completely unconditional approximation
  property with constant\/~$D$, then for every\/~$D'>D$ there is a
  complete approximating sequence of Fourier multipliers\/~$(\ph_k)$
  that realises the completely unconditional approximation property
  with constant\/~$D'$: the Fourier
  multipliers\/~$\sum_{k=1}^n\eps_k\Delta\ph_k$ are uniformly completely
  bounded by\/~$D'$ on\/~$X_\La$.
\end{prp}

Let us now describe how to skip blocks in an approximating sequence in
order to construct an operator that acts like the Riesz projection on
the sumset of two infinite sets. The following trick will be used in
the induction below (compare with the proof of \cite[Theorem~4.2]{nptv}):
\begin{equation*}
  \begin{pmatrix}
    1&1\hskip\arraycolsep\vrule\hskip\arraycolsep0\\0&1\hskip\arraycolsep\vrule\hskip\arraycolsep0\\\noalign{\hrule}0&0\hskip\arraycolsep\vrule\hskip\arraycolsep0
  \end{pmatrix}
  -
  \begin{pmatrix}
    1&1\hskip\arraycolsep\vrule\hskip\arraycolsep0\\1&1\hskip\arraycolsep\vrule\hskip\arraycolsep0\\1&1\hskip\arraycolsep\vrule\hskip\arraycolsep0
  \end{pmatrix}
  +
  \begin{pmatrix}
    1&1&1\\1&1&1\\1&1&1
  \end{pmatrix}
  =
  \begin{pmatrix}
    1&1&1\\0&1&1\\0&0&1
  \end{pmatrix}.
\end{equation*}
\begin{lem}
  \label{diag}
  Let\/ $\Ga$~be a discrete group and\/~$\La\subseteq\Ga$.  Suppose that\/
  $\La$ contains the sumset\/~$\Row\Col$ of two infinite sets\/
  $\Row$~and\/~$\Col$. Let\/ $(T_k)$~be either an approximating sequence
  on\/~$\Ell^p_\La$ with\/ $p\in\iod[1,\infty]$, or an approximating
  sequence of Fourier multipliers on\/~$\Cont_\La$. Let\/~$\eps>0$. There
  is a sequence\/~$(\row_i)$ in\/~$\Row$, a sequence\/~$(\col_i)$ in\/~$\Col$,
  and there are indices\/~$l_1<k_2<l_2<k_3<\dots$ such that, for
  every\/~$n$, the \emph{skipped block sum}
  \begin{equation}
    \label{sbs}
    U_n=T_{l_1}+(T_{l_2}-T_{k_2})+\dots+(T_{l_n}-T_{k_n})
  \end{equation}
  acts, up to\/~$\eps$, as the Riesz projection on the sumset\/~$\{\row_i\col_j\}_{i,j\le n}$:
  \begin{equation*}
    \begin{cases}
      \norm{U_n(\la_{\row_i\col_j})-\la_{\row_i\col_j}}<\eps&\text{if\/ $i\le j\le n$},\\
      \norm{U_n(\la_{\row_i\col_j})}<\eps&\text{if\/ $j<i\le n$}.
    \end{cases}
  \end{equation*}
\end{lem}
\begin{proof}
  Let us construct the sequences and indices by induction. If~$n=1$,
  let $\row_1$~and~$\col_1$ be arbitrary; there is~$l_1$ such that
  $\norm{T_{l_1}(\la_{\row_1\col_1})-\la_{\row_1\col_1}}<\eps$.
  Suppose that $\row_1,\dots,\row_n$, $\col_1,\dots,\col_n$,
  $l_1,\dots,l_n$, and $k_2,\dots,k_n$ have been
  constructed. Let~$\delta>0$ be chosen later.
  \begin{itemize}
  \item The operator~$U_n$ defined by Equation~\eqref{sbs} has finite
    rank. If it is a Fourier multiplier, one can choose an
    element~$\row_{n+1}\in \Row$ such that
    $U_n(\la_{\row_{n+1}\col_j})=0$ for~$j\le n$. If it acts
    on~$\Ell^p_\La$ with $p\in\iod[1,\infty]$, one can choose an
    element~$\row_{n+1}\in \Row$ such that
    $\norm{U_n(\la_{\row_{n+1}\col_j})}<\delta$ for~$j\le n$ because
    $(\la_\ga)_{\ga\in\Ga}$ is weakly null in $\Ell^p$.
  \item There is~$k_{n+1}>l_n$ such that
    $\norm{T_{k_{n+1}}(\la_\ga)-\la_\ga}<\delta$ for~$\ga\in\{\row_i\col_j:1\le i\le n+1,1\le j\le n\}$.
  \item Again, choose~$\col_{n+1}\in \Col$ such that $\norm{(U_n-T_{k_{n+1}})(\la_{\row_i\col_{n+1}})}<\delta$ for~$i\le n+1$.
  \item Again, choose~$l_{n+1}>k_{n+1}$ such that $\norm{T_{l_{n+1}}(\la_\ga)-\la_\ga}<\delta$ for~$\ga\in\{\row_i\col_j:1\le i,j\le n+1\}$.
  \end{itemize}
  Let~$U_{n+1}=U_n+(T_{l_{n+1}}-T_{k_{n+1}})$. If $i\le n+1$ and~$j\le n$, then
  \begin{equation*}
    \norm{\Delta U_{n+1}(\la_{r_ic_j})}\le\norm{T_{l_{n+1}}(\la_{r_ic_j})-\la_{r_ic_j}}+\norm{\la_{r_ic_j}-T_{k_{n+1}}(\la_{r_ic_j})}<2\delta,
  \end{equation*}
  so that
  \begin{align*}
    \norm{U_{n+1}(\la_{r_ic_j})-\la_{r_ic_j}}<\eps+2\delta&\quad\text{if $i\le j\le n$}\\
    \norm{U_{n+1}(\la_{r_ic_j})}<\eps+2\delta&\quad\text{if $j<i\le n$}\\
    \norm{U_{n+1}(\la_{r_{n+1}c_j})}<3\delta&\quad\text{if $j\le n$.}
  \end{align*}
If $i\le n+1$, then
\begin{multline*}
  \norm{U_{n+1}(\la_{r_ic_{n+1}})-\la_{r_ic_{n+1}}}
  \\\le\norm{(U_n-T_{k_{n+1}})(\la_{r_ic_{n+1}})}+\norm{T_{l_{n+1}}(\la_{r_ic_{n+1}})-\la_{r_ic_{n+1}}}
  <2\delta.
\end{multline*}
This shows that our choice of~$\row_{n+1},\col_{n+1},k_{n+1}$ and~$l_{n+1}$ is adequate if
$\delta$~is small enough.
\end{proof}

This construction will provide an obstacle to the unconditionality of sumsets.

\begin{thm}
  \label{T}
  Let\/ $\Ga$~be a discrete group and\/~$\La\subseteq\Ga$.  Suppose that\/ $\La$
  contains the sumset\/~$\Row\Col$ of two infinite sets\/ $\Row$~and\/~$\Col$. 
  \begin{enumerate}
  \item Let\/ $1<p<\infty$. The complete unconditional constant of any approximating
    sequence for\/~$\Ell^p$ is bounded below by the norm of the
    Riesz projection on\/~$\Sch^p$, and thus by\/~$\csc\pi/p$.
  \item The spaces\/ $\Ell^1_\La$~and\/~$\Cont_\La$ do not enjoy the complete
    unconditional approximation property. 
  \item If\/ $\Ga$~is amenable, then the space\/~$\Cont_\La$ does not enjoy the unconditional
    approximation property.
  \end{enumerate}
\end{thm}

\begin{proof}
  Let $(T_k)$~be an approximating sequence on~$\Ell^p_\La$.
  By Lemma~\ref{diag}, for every~$\eps>0$ and every~$n$, there are
  elements~$r_1,\dots,r_n\in\Row$, $c_1,\dots,c_n\in\Col$ such that
  the Fourier multiplier~$\ph$ given by the indicator function
  of ${\{r_ic_j\}_{i\le j}}$ is near to a skipped block sum~$U_n$
  of~$(T_k)$ in the sense that
  $\|U_n(\la_{r_ic_j})-\ph_{r_ic_j}\la_{r_ic_j}\|<\eps$. But $U_n$~is
  the mean of two operators of the form~\eqref{cuap:eq}: its complete
  norm will provide a lower bound for the complete unconditional
  constant of~$X_\La$. Let us repeat the argument of
  Lemma~\ref{trans1} with~$x\in\Sch^p_n$. As
  \begin{multline*}
    \Bignorm{\sum_{i,j=1}^nx_{i,j}\e_{i,j}}_{\Sch^p_n}\\
    \begin{aligned}
      &=\Bignorm{\Bigl(\sum_{i=1}^n\e_{i,i}\otimes\la_{r_i}\Bigr)\Bigl(\sum_{i,j=1}^nx_{i,j}\e_{i,j}\otimes\la_\id\Bigr)\Bigl(\sum_{j=1}^n\e_{j,j}\otimes\la_{c_j}\Bigr)}_{\Ell^p(\tr\otimes\tau)}\\
      &=\Bignorm{\sum_{i=1}^nx_{i,j}\e_{i,j}\otimes\la_{r_ic_j}}_{\Ell^p(\tr\otimes\tau)}
    \end{aligned}
  \end{multline*}
  and
  \begin{equation*}
    \Bignorm{\sum_{i=1}^nx_{i,j}\e_{i,j}\otimes(U_n(\la_{r_ic_j})-\ph_{r_ic_j}\la_{r_ic_j})}_{\Ell^p(\tr\otimes\tau)}<n^2\eps\norm{x}_{\Sch^p_n},
  \end{equation*}
  the complete norm of~$U_n$ is nearly bounded below by the norm of
  the Riesz projection on~$\Sch^p_n$:
  \begin{align*}
    \Bignorm{\sum_{i=1}^nx_{i,j}\e_{i,j}\otimes U_n(\la_{r_ic_j})}_{\Ell^p(\tr\otimes\tau)}
    &>\Bignorm{\sum_{i\le j}x_{i,j}\e_{i,j}\otimes\la_{r_ic_j}}_{\Ell^p(\tr\otimes\tau)}-n^2\eps\norm{x}_{\Sch^p_n}\\
    &=\norm{\Rt(x)}_{\Sch^p_n}-n^2\eps\norm{x}_{\Sch^p_n}.
  \end{align*}
  This proves~$(a)$ as well as the first assertion in~$(b)$, because
  the Riesz projection is unbounded on~$\Sch^1$.  Let $(T_k)$~be an
  approximating sequence on~$\Cont_\La$; by Lemma~\ref{fm}, we may
  suppose that $(T_k)$~is a sequence of Fourier multipliers. Thus the
  second assertion in~$(b)$ follows from Lemma~\ref{diag} combined
  with the preceding argument (where $\Sch^p_n$~is replaced
  by~$\mat_n$ and $\Ell^p(\tr\otimes\tau)$ by~$\mat_n\otimes\Cont$)
  and the unboundedness of the Riesz projection
  on~$\Sch^\infty$. For~$(c)$, note that the Fourier multipliers $T_k$
  are automatically c.b.\ on~$\Cont_\La$ if $\Ga$~is amenable
  (proof of Theorem~\ref{trans2}$\,(c)$).
\end{proof}

Theorem~\ref{T}$\,(b)$ was originally devised to prove that the
Hardy space $\mathrm H^1$, corresponding to the case
$\La=\N\subseteq\Z$ and $p=1$, admits no completely unconditional
basis (see \cite{ri00,ri01}).  Theorem~\ref{T}$\,(c)$ both generalises
the fact that a sumset cannot be a Sidon set (see
\cite[\S\S\,1.4,\,6.6]{lr75} for two proofs and historical remarks, or
\cite[Proposition~IV.7]{lq04}) and Daniel Li's result
\cite[Corollary~13]{li96} that the space~$\Cont_\La$ does not have the
``metric'' unconditional approximation property if $\Ga$~is abelian
and $\La$ contains a sumset.  Li (\cite[Theorem~10]{li96}) also
constructed a set $\La\subseteq\Z$ such that $\Cont_\La$ has this
property, while $\La$ contains the sumset of arbitrarily large sets.
This theorem also provides a new proof that the disc algebra has no
unconditional basis and answers \cite[Question~6.1.6]{ne99}.

\begin{exa}
  Neither the span of products~$\{r_ir_j\}$ of two Rademacher
  functions in the space of continuous functions on~$\{-1,1\}^\infty$ nor the span of
  products~$\{s_is_j\}$ of two Steinhaus functions
  in the space of continuous functions on~${\T}^\infty$ has an unconditional basis.
\end{exa}

\section{Relative Schur multipliers of rank one}\label{sec:rsm1}

Let $\rh$ be an \emph{elementary} Schur multiplier on $\Sch^\infty$, that is,
\[\rh=x\otimes y=(x_ry_c)_{(\row,\col)\in\Row\times\Col}.\] Then its
norm is
$\sup_{\row\in\Row}\abs{x_\row}\sup_{\col\in\Col}\abs{y_\col}$. How is
this norm affected if $\rh$ is only partially specified, that is, if
the action of~$\rh$ is restricted to matrices with a given support?

\begin{thm}
  \label{mps}
  Let\/ $I\subseteq\Row\times\Col$ and consider\/
  $(x_\row)_{\row\in\Row}$~and\/~$(y_\col)_{\col\in\Col}$. The relative
  Schur multiplier on\/~$\Sch^\infty_I$ given by\/ $(x_\row
  y_\col)_{(\row,\col)\in I}$ has norm\/~$\sup_{(\row,\col)\in
    I}\abs{x_\row y_\col}$.
\end{thm}

Note that the norm of the Schur multiplier $(x_\row
y_\col)_{(\row,\col)\in I}$ is bounded by
\begin{math}
  \sup_{\row\in\Row}\abs{x_\row}\*\sup_{\col\in\Col}\abs{y_\col}
\end{math}
because the matrix $(x_\row y_\col)_{(\row,\col)\in\Row\times\Col}$ is a
trivial extension of $(x_\row y_\col)_{(\row,\col)\in I}$; the proof
below provides a constructive nontrivial extension of this Schur
multiplier that is a composition of ampliations of the Schur
multiplier in the following lemma.

\begin{lem}
  \label{schur}
  The Schur multiplier\/
  \begin{math}
    \begin{pmatrix} \overline z&w\\ \overline w&z\end{pmatrix}
  \end{math}
  has norm\/~$\max(\abs{z},\abs{w})$ on\/~$\mat_2$.
\end{lem}
\begin{proof}
  This follows from the decomposition
  \begin{equation*}
    \begin{pmatrix} \overline z&w\\ \overline w&z\end{pmatrix}
    =\frac{\abs{z}+\abs{w}}2
    \begin{pmatrix}
      \bar tu\\t\bar u
    \end{pmatrix}
    \otimes
    \begin{pmatrix}
      \overline{tu}&tu
    \end{pmatrix}
    +\frac{\abs{z}-\abs{w}}2
    \begin{pmatrix}
      \bar tu\\-t\bar u
    \end{pmatrix}
    \otimes
    \begin{pmatrix}
      \overline{tu}&-tu
    \end{pmatrix},
  \end{equation*}
where $t,u\in\T$ are chosen so that $z=\abs{z}t^2$ and $w=\abs{w}u^2$.
 \end{proof}

\begin{proof}[Proof of Theorem~\ref{mps}]
  We may suppose that $\Col$~is the finite set~$\{1,\dots,m\}$ and
  that $\Row$~is the finite set~$\{1,\dots,n\}$, that each~$y_\col$ is
  nonzero, and that each row in~$\Row$ contains an element of~$I$.  We
  may also suppose that
  $(\abs{x_\row})_{\row\in\Row}$~and~$(\abs{y_\col})_{\col\in\Col}$
  are nonincreasing sequences. For each~$\row\in\Row$ let $c_r$~be the
  least column index of elements of~$I$ in or above row~$\row$; in
  other words,
  \begin{equation*}
    \col_\row=\min_{\row'\le\row}\min\{\col:(\row',\col)\in I\}.
  \end{equation*}
  The sequence~$(\col_\row)_{\row\in\Row}$ is nonincreasing. Let us
  define its inverse $(\row_\col)_{\col\in\Col}$ in the sense that
  $\row_\col\le\row\Leftrightarrow\col_\row\le\col$. For
  each~$\col\in\Col$, let $\row_\col=\min\{\row:\col_\row\le\col\}$.
  Given~$\row$, let $\row'\le\row$~be such that $(\row',\col_\row)\in
  I$; then $\abs{x_\row
    y_{\col_\row}}\le\abs{x_{\row'}y_{\col_\row}}$, so that
  $\sup_{\row\in\Row}\abs{x_\row y_{\col_\row}}\le\sup_{(\row,\col)\in
    I}\abs{x_\row y_\col}$ and the rank~1 Schur multiplier
  \[\rh_0=(x_\row y_{\col_\row})_{(\row,\col)\in\Row\times\Col}\] with pairwise equal columns is
  bounded by~$\sup_{(\row,\col)\in I}\abs{x_\row y_\col}$ on~$\Sch^\infty_n$.
  We will now ``correct''~$\rh_0$ without increasing its norm so as to make
  it an extension of~$(x_\row y_\col)_{(\row,\col)\in I}$. Let
  $\row\in\Row$~and~$\col'\ge\col_\row$; then
    \begin{align*}
      x_\row y_{\col'}
      =x_\row y_{\col_\row}\frac{y_{\col_\row+1}}{y_{\col_\row}}\cdots\frac{y_{\col'}}{y_{\col'-1}}
      &=x_\row y_{\col_\row}\prod_{\col_\row\le\col\le\col'-1}\frac{y_{\col+1}}{y_{\col}}\\
      &=x_\row y_{\col_\row}\prod_{\substack{\row\ge\row_\col\\\col'\ge\col+1}}\frac{y_{\col+1}}{y_{\col}}.
   \end{align*}
    This shows that it suffices to compose the Schur
    multiplier~$\rh_0$ with the $m-1$ rank~2 Schur multipliers with
    block matrix
    \begin{equation*}
      \rh_{\col}=
      \bordermatrix{%
        &\scriptstyle1~\cdots~\col&&\scriptstyle\col+1~\cdots~m\cr
        \begin{matrix}\scriptstyle1\\\scriptstyle\vdots\\\scriptstyle\row_\col-1\end{matrix}
        &\displaystyle\overline{\left(\frac{y_{\col+1}}{y_\col}\right)}&\vrule&1\cr\noalign{\kern-1pt}
        &\multispan3\hrulefill\cr\noalign{\kern-1pt}
        \hfil\begin{matrix}\scriptstyle\row_\col\\\scriptstyle\vdots\\\scriptstyle n\end{matrix}
        &1&\vrule&\displaystyle\frac{y_{\col+1}}{y_\col}\cr
        },
    \end{equation*}
    each of which has norm~1 on~$\Sch^\infty_n$ by Lemma~\ref{schur}. 
\end{proof}

\begin{rem}
  We learned after submitting this article that Timur Oikhberg proved
  Theorem~\ref{mps} independently and gave some applications to it; see~\cite{oi10}.
\end{rem}
\begin{rem}
  As an illustration, let $\Col=\Row=\{1,\dots,n\}$ and~$I=\{(r,c):r\ge
  c\}$, and let $a_i$~be an increasing sequence of positive numbers.
  Take $x_r=a_r$~and~$y_c=1/a_c$. Then the relative Schur
  multiplier~$(a_r/a_c)_{r\le c}$ has norm~1. The above proof actually
  constructs the norm~1 extension
  $\left(\min(a_r/a_c,a_c/a_r)\right)_{(r,c)}$. If we
  put~$a_i=\e^{x_i}$, we recover that $(\e^{-|x_r-x_c|})_{(r,c)}$~is
  positive definite, that is, $\abs{\cdot}$~is a conditionally
  negative function on~$\R$.
\end{rem}\medskip

\noindent
2010 \emph{Mathematics subject classification}:
Primary 47B49; Secondary 43A22, 43A46, 46B28.\medskip

\noindent
\emph{Key words and phrases}: Fourier multiplier, Toeplitz Schur
multiplier, lacunary set, unconditional approximation property,
Hilbert transform, Riesz projection.


\end{document}